\documentclass{article}

\usepackage{amsmath,amssymb,amsthm}
\usepackage{mathrsfs}
\usepackage[colorlinks=true]{hyperref}
\usepackage{pdfsync}
\usepackage{graphicx,color}
\usepackage{appendix}

\def\R{\mathbb{R}}
\def\N{\mathbb{N}}

\newcommand{\C}{\mathbb{C}}

\usepackage{float} 
\usepackage{indentfirst} 
\usepackage{float} 

\usepackage{threeparttable}
\usepackage{appendix}

\newtheorem{theorem}{Theorem}
\newtheorem{lemma}{Lemma}

\newtheorem{corollary}{Corollary}
\newtheorem{definition}{Definition}
\theoremstyle{definition}\newtheorem{remark}{Remark}

\renewcommand{\geq}{\geqslant}
\renewcommand{\leq}{\leqslant}

\renewcommand{\d}{\displaystyle}

\newcommand{\MTCP}{{\bf (MTCP)}}
\newcommand{\OCPZ}{{\bf (OCP0)}}

\newcommand{\OCPReps}{$\bf (OCPR)_\gamma$\ }

\title{Minimum time control of the rocket attitude reorientation associated with orbit dynamics}

\author{Jiamin Zhu\footnote{Sorbonne Universit\'es, UPMC Univ Paris 06, CNRS UMR 7598, Laboratoire Jacques-Louis Lions, F-75005, Paris, France (\texttt{zhu@ann.jussieu.fr}).}
\and
Emmanuel Tr\'elat\footnote{Sorbonne Universit\'es, UPMC Univ Paris 06, CNRS UMR 7598, Laboratoire Jacques-Louis Lions, Institut Universitaire de France, F-75005, Paris, France (\texttt{emmanuel.trelat@upmc.fr}).}
\and
Max Cerf\footnote{Airbus Defence and Space, Flight Control Unit, 66 route de Verneuil, BP 3002, 78133 Les Mureaux Cedex, France (\texttt{max.cerf@astrium.eads.net}).}
}

\date{}

\begin{document}
\maketitle

\begin{abstract}
In this paper, we investigate the minimal time problem for the guidance of a rocket, whose motion is described by its attitude kinematics and dynamics but also by its orbit dynamics. Our approach is based on a refined geometric study of the extremals coming from the application of the Pontryagin maximum principle. Our analysis reveals the existence of singular arcs of higher-order in the optimal synthesis, causing the occurrence of a chattering phenomenon, i.e., of an infinite number of switchings when trying to connect bang arcs with a singular arc. 

We establish a general result for bi-input control-affine systems, providing sufficient conditions under which the chattering phenomenon occurs. We show how this result can be applied to the problem of the guidance of the rocket. Based on this preliminary theoretical analysis, we implement efficient direct and indirect numerical methods, combined with numerical continuation, in order to compute numerically the optimal solutions of the problem. 
\end{abstract}

\bigskip


\tableofcontents

\section{Introduction}
The optimal control of orbit transfer (see, e.g., \cite{BonnardTrelat,Bryson,Marec}) and attitude reorientation (see, e.g., \cite{Bilimoria,Seywald,SHEN}) for spacecrafts have been extensively studied in the past few decades. 
The optimal control problem of \emph{orbit transfer} focuses mostly on how to move the spacecraft from one orbit or point to another orbit or point by using minimum energy, while the optimal control problem of \emph{attitude reorientation} is mainly devoted to determine how to change the pointing direction of the spacecraft in minimum time. 
In the existing literature, these two optimal control problems are considered separately in general. 
From the engineering point of view, for most satellites, it is appropriate to design separately the control laws for the orbit movement and for the attitude movement. 
However, for the rockets, the trajectory is controlled by its attitude angles: the way to make the rocket follow its nominal trajectory is to change its attitude angles, and therefore it is desirable to be able to determine the optimal control subject to the coupled dynamical system.
Though the control of the coupled problem was also studied in many previous works (see, e.g., \cite{Lara,Guzzetti,Knutson}), it does not seem that the problem has been investigated in the optimal control framework so far.

In this paper, we consider the time minimum control of the attitude reorientation coupled with the orbit dynamics of a rocket, denoted in short $\MTCP$.
The chattering phenomenon that may occur according to the terminal conditions under consideration, makes in particular the problem quite difficult.
\emph{Chattering} means that the control switches an infinite number of times over a compact time interval. Such a phenomenon typically occurs when trying to connect bang arcs with a higher-order singular arc (see, e.g., \cite{FULLER1,Marchal,ZELIKIN,ZTC}). 
In \cite{ZTC}, we studied the planar version of $\MTCP$, where the system consists of a single-input control-affine system, and we established as well the occurence of a chattering phenomenon and that the chattering extremals are locally optimal in $C^0$ topology.\footnote{A trajectory $\bar{x}(\cdot)$ is said to be locally optimal in $C^0$ topology if, for every neighborhood $V$ of $\bar{x}(\cdot)$ in the state space, for every real number $\eta$ so that $| \eta | \leq \epsilon$, for every trajectory $x(\cdot)$, associated to a control $v$ on $[0,T+\eta]$, contained in $W$, and satisfying $x(0) = \bar{x}(0) = x_0$, $x(T+\eta) = \bar{x}(T)$, there holds $C(T+\eta,v) \geq C(T,u)$, where $C$ is the cost functional to be minimized.}

A second important difficulty in $\MTCP$ is due to the coupling of the attitude movement with the orbit dynamics. Indeed the system contains both slow (orbit) and fast (attitude) dynamics. This observation will be particularly important in order to design appropriate numerical approaches.

In order to analyze the extremals of the problem, we use geometric optimal control theory (see \cite{AGRACHEV,SCHATTLER,Trelat2}). 
The Pontryagin maximum principle and the geometric optimal control, especially the concept of Lie bracket, will be used in this paper in order to establish an existence result of the chattering phenomenon.
More precisely, based on the Goh and generalized Legendre-Clebsch conditions, we prove that there exist optimal chattering arcs when trying to connect a regular arc with a singular arc of order two.

There exist various numerical approaches to solve an optimal control problem. The direct methods (see, e.g., \cite{Betts}) consist of discretizing the state and the control and thus of reducing the problem to a nonlinear optimization problem (nonlinear programming) with constraints. The indirect methods consist of numerically solving a boundary value problem obtained by applying Pontryagin maximum principle (PMP, see \cite{Pontryagin}), by means of a shooting method. 
There exist also mixed methods that discretize the PMP necessary conditions and use then a large-scale optimization solver (see, e.g., \cite{Berend}).
Since these numerical approaches are not easy to initialize successfully, it is required them to combine with other theoretical or numerical approaches (see the survey \cite{Trelat2}).
Here, we will use numerical continuation, which has proved to be very powerful tool to be combined with the PMP. For example, in \cite{CHT,Gergaud,Martinon}, the continuation method is used to solve difficult orbit transfer problems. 

However, due to the chattering phenomenon, numerical continuation combined with shooting cannot give an optimal solution to the problem for certain terminal conditions for which the optimal trajectory contains a singular arc of higher-order.
In that case, we propose sub-optimal strategies by using direct methods computing approximate piecewise constant controls. 
It is noticeable that our indirect approach can also be adapted to generate sub-optimal solutions, by stopping the continuation procedure before its failure due to chattering. This approach happens to be faster than the direct approach, and appears as an interesting alternative for practice.

From the engineer point of view, the theoretical analysis as well as the numerical strategies and the way to design them (in particular, the design of the problem of order zero) are strongly based on the fact that the orbit movement is much slower than the attitude movement. 

\medskip

The paper is organized as follows. 
In Section \ref{Chp_model}, we describe the mathematical model of the system consisting of the attitude dynamics, of the attitude kinematics, and of the orbit dynamics. 
In Section \ref{Chp_generalresults}, we recall the Pontryagin maximum principle and some higher necessary conditions of optimality (Goh and generalized Legendre-Clebsch conditions) for bi-input control affine systems. Based on these necessary conditions of optimality, we establish a result on the existence of the optimal chattering extremals. 
In Section \ref{Chp_Gae}, we analyze the regular and singular extremals of the problem $\MTCP$. For the regular extremals, we classify the switching points and state some useful properties. For the singular extremals (which are of order two), we show that the chattering phenomenon occurs for the problem $\MTCP$ by using the results given in the previous section.
In Section \ref{Chp_continuation}, we propose a numerical approach to solve the problem $\MTCP$ by implementing numerical continuation combined with shooting.
Numerical results are given in Section \ref{Chp_numerical}.

\section{Model and problem statement}\label{Chp_model}
The problem is to control the attitude movement coupled with the orbit dynamics in the launching ascent stage for a rocket. 
In this paper, we take the system parameters of the rocket \emph{Ariane 5}. 
In order to keep the stability of the rocket along the flight, the attitude maneuver should be moderate, i.e., at most $\pm 20$ degrees, and then it is possible to use Euler angles to model the attitude of the engine.
In this section, we first define the coordinates systems, and then we give the equations of the attitude dynamics, of the attitude kinematics and of the orbit dynamics. 
The model consists of eight ordinary differential equations: three for the components of the velocity vector, three for the Euler angles and two for the components of the angular velocity vector.

\subsection{Coordinate systems}
Throughout the paper, we make the following assumptions:
\begin{itemize}
\item The Earth is a sphere and is fixed in the inertial space, i.e., the angular velocity of the Earth is zero, which means that $\vec{\omega}_{ei}=\vec{0}$.
\item The position of the rocket remains the same during the maneuver of the rocket.
\item The rocket is an axial symmetric cylinder.
\item The aero-dynamical forces are zero.
\item The rocket engine cannot be shut off during the flight and the module of the thrust force is constant, taking its maximum value, i.e., $T=T_{max}$.
\end{itemize}

The \textbf{unit single-axis rotation} maps $R_i(\sigma)$: $\R \rightarrow \R^{3 \times 3}$, for $\sigma \in \R$, $i=x,y,z$ are defined by
\begin{equation*}
    \label{}
    R_x(\sigma)=
    \begin{pmatrix}
        1  &0                    &0                \\
        0  & \cos \sigma   & \sin \sigma\\
        0  &- \sin \sigma  &  \cos \sigma
    \end{pmatrix}
    , \hspace{0.3cm}
    R_y(\sigma)=
    \begin{pmatrix}
        \cos \sigma   & 0                 & -\sin \sigma   \\
         0                  & 1                  & 0  \\
        \sin \sigma  & 0                    &  \cos \sigma
    \end{pmatrix}
    , \hspace{0.3cm}
    R_z(\sigma)=
    \begin{pmatrix}
        \cos \sigma   & \sin \sigma   & 0 \\
        - \sin \sigma &  \cos \sigma & 0\\
        0                  & 0                   & 1
    \end{pmatrix}.
\end{equation*}
For a given vector $\vec{e} \in \R^3$, taking $R_i(\sigma) \vec{e}$ means to rotate the vector $\vec{e}$ with respect to the axis $i$ by an angle of $\sigma$. With this definition, we next introduce the coordinate frames that will be used throughout the paper. 

The \textbf{Earth frame} $S_g=(\hat{x}_g,\hat{y}_g,\hat{z}_g)$ is fixed around the center of the Earth $O$. The axis $\hat{z}_g$ points to the North pole, and the axis $\hat{x}_g$ is in the equatorial plan of the Earth pointing to the equinox.

The \textbf{launch frame} $S_{R}=(\hat{x}_R,\hat{y}_R,\hat{z}_R)$ is fixed around the launch point $O_R$ (where the rocket is launched). The axis $\hat{x}_R$ is normal to the local tangent plane, pointing to the launch direction (here we assume that the rocket is vertically launched, i.e., the launch direction is perpendicular with the local tangent plane), and the axis $\hat{z}_R$ points to the North. As shown in Figure \ref{Frames} (a), the launch frame is derived from the Earth frame by two ordered unit single-axis rotations $R_z(\ell_R)$ and $R_y(-L_R)$,
\begin{equation*}
S_g \xrightarrow{R_z(\ell_R)} \circ \xrightarrow{R_y(-L_R)}  S_R
\end{equation*}
where $\ell_R$ and $L_R$ are the longitude and latitude of the launch point, respectively.
\begin{figure}[h] 
\centering
\includegraphics[scale=0.9]{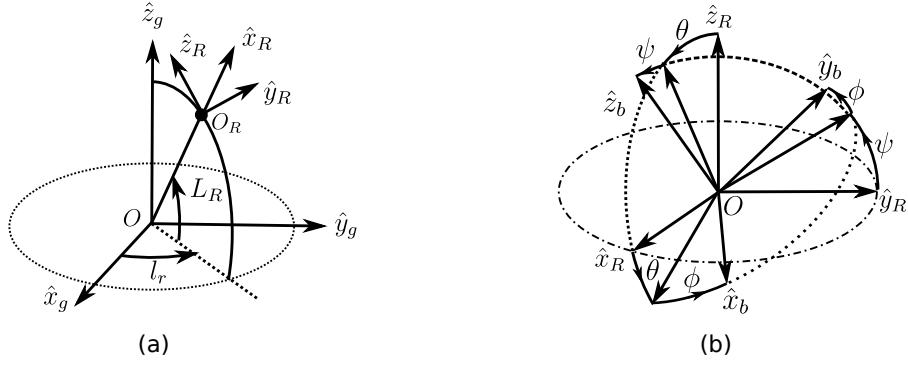}           
\caption{Coordinate systems and relations.}
\label{Frames}
\end{figure}

The \textbf{body frame} $S_b=(\hat{x}_b,\hat{y}_b,\hat{z}_b)$ is defined as follows. The origin of the frame $O_b$ is fixed around the mass center of the rocket, the axis $\hat{z}_b$ is along the axis-symmetric axis of the rocket, and the axis $\hat{x}_b$ is in the cross-section. The body frame can be derived by three ordered unit single-axis rotations from the launch frame, as shown in Figure \ref{Frames} (b),
$$
    S_{R} \xrightarrow{R_y(\theta)} \circ \xrightarrow{R_x(\psi)} \circ \xrightarrow{R_z(\phi)} S_b
$$
where $\theta$ is the pitch angle, $\psi$ is the yaw angle and $\phi$ is the roll angle. Therefore, the transformation matrix from $S_R$ to $S_b$ is
\begin{equation} \label{LbR}
\begin{split}
	L_{bR}=&R_z(\phi) R_x(\psi) R_y(\theta)  \\
	=& \begin{pmatrix}
      \cos \theta \cos \phi + \sin \theta \sin \psi \sin \phi & \cos \psi \sin \phi & -\sin \theta 		\cos \phi + \cos \theta \sin \psi \sin \phi    \\
      -\cos \theta \sin \phi + \sin \theta \sin \psi \cos \phi & \cos \psi \cos \phi &\sin \theta 		\sin \phi + \cos \theta \sin \psi \cos \phi \\
      \sin \theta \cos \psi & -\sin \psi & \cos \theta \cos \psi
\end{pmatrix},
\end{split}
\end{equation}
and the transformation matrix from $S_b$ to $S_R$ is $L_{Rb}=L_{bR}^{-1}=L_{bR}^{\top}$.

\subsection{Attitude dynamic equations}
The attitude dynamics are written in vectorial form in the body frame $S_b$ as
\begin{equation} \label{ade_vf}
    (I \vec{\omega})_b = -(\vec{\omega})_b \wedge (I \vec{\omega})_b + (\vec{M})_b,
\end{equation}
where $I$ is the inertia matrix, $\vec{\omega}$ is the absolute angular velocity vector, i.e., the angular velocity of the rocket with respect to the inertial space, and $\vec{M}$ is the control torques introduced by the rocket thrust. The index $(\cdot)_b$ means that the vectors are expressed in the body frame $S_b$. 

Setting $(I)_b=\mathrm{diag}(I_x,I_y,I_z)$, $(\vec{\omega})_b=(\omega_x,\omega_y,\omega_z)^{\top}$ and $(\vec{M})_b = (M_x,M_y,M_z)^{\top}$, \eqref{ade_vf} gives
\begin{equation} \label{ade_cf}
\begin{cases}
	\d{I_x \dot{\omega}_x=(I_y-I_z) \omega _y \omega _z + M_x},\\
	\d{I_y \dot{\omega}_y=(I_z-I_x) \omega _x \omega _z + M_y},\\
	\d{I_z \dot{\omega}_z=(I_x-I_y) \omega _x \omega _y + M_z}.
\end{cases}
\end{equation}
The control torque $\vec{M}$ is the cross product of the thrust vector $\vec{T}$ and of its moment arm $\vec{L}$. The moment arm is the vector from the center of mass $O_b$ to the force acting point $O_F$, given here by $ (\vec{L})_b =(0, 0, -l)^\top$. Moreover, as shown in Figure \ref{Thrust} (a), the thrust force vector is 
$$
\vec{T} = (-T \sin \mu \cos \zeta , -T \sin \mu \sin \zeta, T \cos \mu)^\top,
$$
where $T = T_{max}$, $\mu \in [0,\mu_{max}]$, and $\zeta \in [- \pi,\pi]$. The control torque is then
$$
 (\vec{M})_b = (\vec{L})_b \wedge (\vec{T})_b= ( -T l \sin \mu \sin \zeta ,\,T l \sin \mu \cos \zeta ,\, 0 )^\top.
$$
By assumption, the rocket is axial symmetric, and hence $I_x = I_y$. Assume that $\d{\omega_z(0)=0}$, and let $\d{ b = T_{max} l / I_x }$. Then \eqref{ade_cf} gives
$$
\dot{\omega}_x= - b \sin \mu \sin \zeta,\qquad \dot{\omega}_y=   b \sin \mu \cos \zeta,
$$
with $\omega_z \equiv 0$.

\begin{figure}[h]
\centering
\includegraphics[scale=0.8]{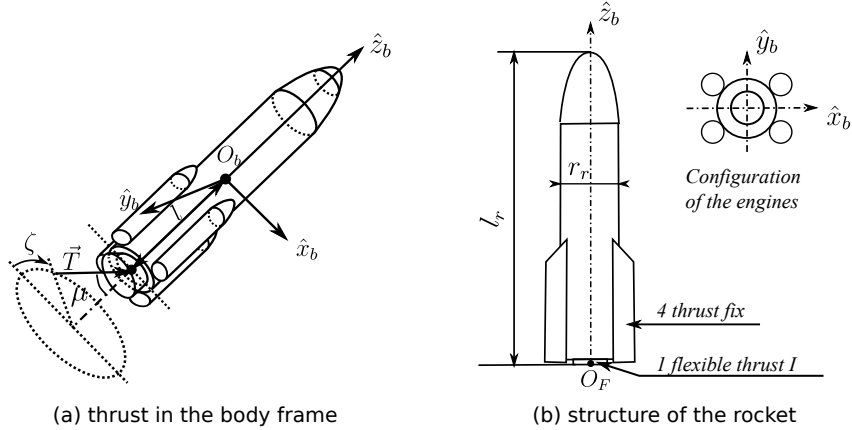}
\caption{Thrust in the body Frame}
\label{Thrust}
\end{figure}

According to the parameters of the rocket engine, $\mu_{max}$ is less than $10$ degrees
and thus the error between $\sin \mu$ and $\mu$ is less than $0.5 \%$.
Therefore, in the model we make the approximation $\sin \mu \simeq \mu$ and we define $u_1= \bar{\mu} \cos \zeta$ and $u_2= \bar{\mu} \sin \zeta$ and $\bar{\mu}=\mu/\mu_{max}$ with $\bar{b}=b \mu_{max}$. Hence
\begin{equation} \label{attitude_dynamics}
\dot{\omega }_x= -\bar{b} u_2,\qquad \dot{\omega }_y=  \bar{b} u_1.
\end{equation}

\subsection{Attitude kinematics equations}
Since $\vec{\omega}$ is the angular velocity vector of the rocket with respect to the inertial space, it is equal to the sum of the angular velocity $\vec{\omega}_{bg}$ of the rocket with respect to the Earth frame, of the angular velocity $\vec{\omega}_{ge}$ of the Earth frame with respect to the Earth, and of the angular velocity $\vec{\omega}_{ei}$ of the Earth with respect to the inertial space. According to the assumptions and definitions of the frames, it is easy to see that the last two terms are zero, and thus $\vec{\omega} = \vec{\omega}_{bg}$. Therefore, based on the definition of the body frame, the relationship between angular velocity and Euler angles are
\begin{equation*} \label{}
    \begin{pmatrix}
        \omega_x    \\
        \omega_y  \\
        \omega_z
    \end{pmatrix}
    = L_{bR} \begin{pmatrix}
        0   \\
        \dot{\theta}\\
        0
    \end{pmatrix}
    + \begin{pmatrix}
        \dot{\psi} \cos \phi    \\
        -\dot{\psi} \sin \phi    \\
        0
    \end{pmatrix}
    +\begin{pmatrix}
        0    \\
        0  \\
        \dot{\phi}
    \end{pmatrix},
\end{equation*}
where $L_{bR}$ is given by \eqref{LbR}. Then the equations of the attitude kinematics are
\begin{equation} \label{attitude_kinetics}
\dot{\theta}=(\omega_x \sin \phi + \omega_y \cos \phi)/ \cos \psi , \quad
\dot{\psi}=\omega_x \cos \phi - \omega_y \sin \phi ,\quad
\dot{\phi}=\tan \psi (\omega_x \sin \phi + \omega_y \cos \phi) .
\end{equation}
Therefore, the two equations of \eqref{attitude_dynamics} and the three equations of \eqref{attitude_kinetics} describe the attitude movement.

Note that when $\psi = \pi/2+k\pi$, $k \in \mathbb{N}$, the Euler angles defined above are not well defined (usual singularities of the Euler angles). We assume in this paper that the maneuvers are small enough, so that these singularities will not be encountered.

\subsection{Orbit dynamics equations}\label{C_ode}
The equation of the orbit dynamics in vectorial form is
\begin{equation} \label{orbit_dynamics_vf}
     \frac{d (\vec{V})_R}{dt} = (\vec{g})_R + \frac{L_{Rb} (\vec{T})_b}{m} + (\vec{\omega})_R \wedge (\vec{V})_R
     - 2 (\vec{\omega}_{ei} \wedge \vec{V})_R - (\vec{\omega}_{ei} \wedge (\vec{\omega}_R \wedge \vec{r}))_R ,
\end{equation}
where the notation $(\cdot)_R$ means that the vector is expressed in the launch frame $S_R$. The vector $\vec{V}$ is the velocity of the rocket with respect to the Earth frame $S_g$, and its components in the launch frame are $v_x$, $v_y$ and $v_z$. The vector $(\vec{g})_R=(g_x,g_y,g_z)^{\top}$ can be approximated by $(\vec{g})_R \approx (-g_0, 0, 0)^\top$, where $g_0$ is a real number representing the standard gravity ($g_0=9.8$).

Due to the fact that the control angle $\mu$ is very small in practice (physical constraints imposed by the rocket engine), we assume that the thrust force is along the body axial symmetric axis. According to the previous assumptions, the equation of the orbit dynamics \eqref{orbit_dynamics_vf} becomes
\begin{equation}\label{orbit_dynamics}
\dot{v}_x= a \sin \theta \cos \psi + g_x ,\qquad \dot{v}_y= - a \sin \psi + g_y,\qquad \dot{v}_z= a \cos \theta \cos \psi + g_z,
\end{equation}
where $a=T_{max}/m$ is constant. 
Note that this additional assumption is made also because the attitude of the rocket is controlled by only a part of the rocket engines, and so the total thrust remains almost parallel to the rocket symmetric axis.

\subsection{Minimum time control problem $\MTCP$}
\paragraph{Model.}
The system \eqref{attitude_dynamics}-\eqref{attitude_kinetics}-\eqref{orbit_dynamics} has two control inputs $u_1$ and $u_2$, and we obtain the system
\begin{equation}\label{sys_full}
\begin{split}
& \dot{v}_x= a \sin \theta \cos \psi + g_x , \qquad \dot{v}_y= - a \sin \psi + g_y,\qquad \dot{v}_z= a \cos \theta \cos \psi + g_z,\\
& \dot{\theta}=(\omega_x \sin \phi + \omega_y \cos \phi)/ \cos \psi , \qquad \dot{\psi}=\omega_x \cos \phi - \omega_y \sin \phi ,\qquad \dot{\phi}= (\omega_x \sin \phi + \omega_y \cos \phi) \tan \psi ,\\
& \dot{\omega }_x= -\bar{b} u_2,\qquad \dot{\omega }_y=  \bar{b} u_1.
\end{split}
\end{equation}
Defining the state variable $x=(v_x, v_y, v_z, \theta, \psi, \phi, \omega_x, \omega_y)$, we write the system \eqref{sys_full} as the bi-input control-affine system
\begin{equation} \label{sys_multi_affine}
    \dot{x} = f(x) + u_1g_1 (x)  + u_2g_2 (x) ,
\end{equation}
where the controls $u_1$ and $u_2$ satisfy the constraint $u_1^2+u_2^2 \leq 1$,
and the vector fields $f$, $g_1$ and $g_2$ are defined by
\begin{multline}\label{fetgi}
f= (a \sin \theta \cos \psi + g_x) \frac{\partial}{\partial v_x}
                 + (- a \sin \psi + g_y) \frac{\partial}{\partial v_y}
                 + (a \cos \theta \cos \psi + g_z) \frac{\partial}{\partial v_z}\\
                 + (\omega_x \sin \phi + \omega_y \cos \phi)/ \cos \psi \frac{\partial}{\partial \theta}
                 + (\omega_x \cos \phi - \omega_y \sin \phi)\frac{\partial}{\partial \psi}
                + \tan \psi (\omega_x \sin \phi + \omega_y \cos \phi) \frac{\partial}{\partial \phi} , \\
g_1 = \bar{b} \frac{\partial}{\partial \omega_y},\qquad g_2 = - \bar{b} \frac{\partial}{\partial \omega_x}.
\end{multline}

\paragraph{Terminal conditions and system parameters.}
Let ${v_{x_0}}$, ${v_{y_0}}$, ${v_{z_0}}$, $\theta_0$, $\psi_0$, $\phi_0$, ${\omega_{x_0}}$, ${\omega_{y_0}}$, $\theta_f$, $\psi_f$, $\phi_f$, $\omega_{x_f}$ and $\omega_{y_f}$ be real numbers. 
The initial conditions are fixed to
\begin{equation} \label{OCPc_ic}
\begin{split}
& v_x(0) = {v_{x_0}},\quad  
v_y(0) = {v_{y_0}},\quad  
v_z(0) = {v_{z_0}},\\  
& \theta(0) = \theta_0,\quad \psi(0)=\psi_0,\quad \phi(0)=\phi_0,\quad
\omega_x(0) = {\omega_{x_0}},\quad \omega_y(0)={\omega_{y_0}}.
\end{split}
\end{equation}
The desired final velocity is required to be parallel to the body axis $\hat{z}_b$, according to $
(\vec{V}(t_f))_R \wedge (\hat{z}_b (t_f))_R=\vec{0}$, and therefore, the constraints on the final conditions are
\begin{equation} \label{OCPc_fc}
\begin{split}
& v_{z_f} \sin \psi_f + v_{y_f} \cos \theta_f \cos \psi_f =0,\quad v_{z_f} \sin \theta_f - v_{x_f} \cos \theta_f =0,\\
& \theta(t_f)=\theta_f,\quad \psi(t_f)=\psi_f,\quad \phi(t_f)=\phi_f, \quad \omega_x(t_f)=\omega_{x_f},\quad \omega_y(t_f)=\omega_{y_f}.
\end{split}
\end{equation}
Note that the parallel condition on the final velocity is due to the fact that most rockets are planned to maintain a zero angle of attack along the flight. The angle of flight, when the air wind is set to zero, is defined as the angle between the velocity and the rocket body axis.

\paragraph{Minimum Time Control Problem $\MTCP$.}
We set $x_0 = ({v_{x_0}},{v_{y_0}},{v_{z_0}},\theta_0,\psi_0,\phi_0,{\omega_{x_0}},{\omega_{y_0}}) \in \R^8$, and we define the target set (submanifold of $\R^8$)
\begin{equation*}
\begin{split}
M_1 = & \{(v_x,v_y,v_z,\theta,\psi,\phi,\omega_x,\omega_y) \in \R^8 \ \mid\  v_z \sin \psi_f + v_y \cos \theta_f \cos \psi_f =0,\\
 &\qquad v_z \sin \psi_f + v_y \cos \theta_f \cos \psi_f =0, \quad \theta =\theta_f,\quad \psi=\psi_f,\quad \phi=\phi_f,\quad \omega_x=\omega_{x_f},\quad \omega_y=\omega_{y_f} \} .
\end{split}
\end{equation*}
The minimum time control problem $\MTCP$ consists of steering the bi-input control-affine system \eqref{sys_multi_affine} from $x(0)=x_0$ to the final target $M_1$ in minimum time $t_f$, with controls satisfying the constraint $u_1^2+u_2^2 \leq 1$.

\section{Some general results for bi-input control-affine systems}\label{Chp_generalresults}
In this section, we focus on the chattering phenomenon for bi-input control-affine systems with control constraints and with commuting controlled vector fields. The results that we are going to give are general and will be used in the next section to analyze the problem $\MTCP$. 

We consider the following general framework.
Let $M$ be a smooth manifold of dimension $n$, let $x_0\in M$ be arbitrary, and let $M_1$ be a submanifold of $M$. We consider on $M$ the minimal time control problem
\begin{equation} \label{pb_ocp}
\left\{ \begin{split}
	& \min t_f , \\
	& \dot{x}(t) = f(x(t))+u_1(t) g_1(x(t))+u_2(t) g_2(x(t)),\quad u=(u_1,u_2) \\
	& \Vert u(t)\Vert^2 = u_1(t)^2+u_2(t)^2 \leq 1 , \\
	& x(0) = x_0,\ x(t_f) \in M_1 , \quad t_f\geq 0\ \textrm{free},
\end{split}\right.
\end{equation}
where $f$, $g_1$ and $g_2$ are smooth vector fields on $M$. 

According to classical results (see, e.g., \cite{Cesari,Trelat2}), there exists at least one optimal solution $(x(\cdot),u(\cdot))$, defined on $[0,t_f]$.

\subsection{Application of the Pontryagin maximum principle}
According to the Pontryagin maximum principle (in short, PMP, see \cite{Pontryagin}), there must exist an absolutely continuous mapping $p(\cdot)$ defined on $[0,t_f]$ (called adjoint vector), such that $p(t)\in T^*_{x(t)}M$ (cotangent space) for every $t\in[0,t_f]$, and a real number $p^0 \leq 0$, with $(p(\cdot),p^0)\neq 0$, such that
$$
\dot{x}(t) = \frac{\partial H}{\partial p}(x(t),p(t),p^0,u(t)),\quad
\dot{p}(t) = -\frac{\partial H}{\partial x}(x(t),p(t),p^0,u(t)) ,
$$
almost everywhere on $[0,t_f]$, where $H(x,p,p^0,u) = h_0(x,p)+u_1 h_1(x,p) +u_2 h_2(x,p)+p^0$ is the Hamiltonian of the optimal control problem \eqref{pb_ocp}.
Here, we have set $h_0(x,p)=\langle p, f(x) \rangle $, $h_1(x,p)=\langle p, g_1(x) \rangle $, and $h_2(x,p)=\langle p, g_2(x) \rangle $. The maximization condition of the PMP yields, almost everywhere on $[0,t_f]$,
\begin{equation} \label{eq_ut}
u(t) = \frac{(h_1(t),h_2(t)) }{\sqrt{h_1(t)^2+h_2(t)^2}} = \frac{\Phi(t)}{\Vert\Phi(t)\Vert},
\end{equation}
whenever $\Phi(t)=(h_1(t),h_2(t))\neq (0,0)$. We call $\Phi$ (as well as its components) the switching function. Note that $\Phi$ is continuous. Here and throughout the paper, we denote by $h_i(t)=h_i(x(t),p(t))$, with a slight abuse of notation.

Moreover, we have the transversality condition $p(t_f) \perp T_{x(t_f)} M_1$, where $T_{x(t_f)}M_1$ is the tangent space to $M_1$ at the point $x(t_f)$, and, the final time $t_f$ being free and the system being autonomous, we have also 
$h_0(x(t),p(t))+\Vert\Phi(t)\Vert+p^0=0,\: \forall t\in[0,t_f]$.

The quadruple $(x(\cdot),p(\cdot),p^0,u(\cdot))$ is called an extremal lift of $x(\cdot)$.
An extremal is said to be normal (resp., abnormal) if $p^0 < 0$ (resp., $p^0 = 0$). 

We say that an arc (restriction of an extremal to a subinterval $I$) is \emph{regular} if $\Vert \Phi(t)\Vert \neq 0$ along $I$. Otherwise, the arc is said to be \emph{singular}.
Note that a singular extremal may be both normal or abnormal. We will see in Section \ref{ext_sin_chat} that the singular extremals of the problem $\MTCP$ must be normal.

A \emph{switching time} is a time $t$ at which $\Phi(t)=(0,0)$, that is, both $h_1$ and $h_2$ vanish at time $t$.
An arc that is a concatenation of an infinite number of regular arcs is said to be \emph{chattering}. The chattering arc is associated with a \emph{chattering control} that switches an infinite number of times, over a compact time interval.
A junction between a regular arc and a singular arc is said to be a \emph{singular junction}.

\subsection{Computation of singular arcs, and necessary conditions for optimality}\label{sec32}
We next define the order of a singular control, since it is important to understand and explain the occurence of chattering. This concept is related to the way singular controls are computed, and since it is a bit technical to define, we start with a preliminary quite informal discussion.
Here and throughout the paper, we use the notation $\mathrm{ad}f.g=[f,g]$ (Lie bracket of vector fields) and $\mathrm{ad}h_i.h_j=\{h_i,h_j\}$ (Poisson bracket of Hamiltonian functions).

\paragraph{Preliminary informal discussion.}
In order to compute singular controls, the usual method is to differentiate several times the switching function, until the control appears in a nontrivial way. If $\Vert \Phi(t)\Vert = 0$ for every $t\in I$, then $h_1(t)=h_2(t)=0$, and, differentiating in $t$, we get, using the Poisson bracket, $\dot{h}_1 = \{ h_0,h_1 \} + u_2 \{ h_2,h_1 \}  = 0$ and $\dot{h}_2 = \{ h_0,h_2 \} + u_1 \{ h_1,h_2 \}  = 0$ along $I$. According to the Goh condition (see \cite{Goh}, see also below), if the singular arc is optimal, then the Goh condition $ \{ h_1,h_2 \}  = \langle p, [g_1,g_2](x)\rangle= 0$ must be satisfied along $I$. Therefore we get that $\dot{h}_1 = \{ h_0,h_1 \}  = \langle p, [f,g_1](x)\rangle=0$ and $\dot{h}_2 = \{ h_0,h_2 \}  = \langle p, [f,g_2](x)\rangle=0$ along $I$.

Let us now assume that the vector fields $g_1$ and $g_2$ commute, i.e., $[g_1,g_2] = 0$.
By differentiating again, we get
\begin{align*}
\ddot{h}_1 = \{ h_0,\{h_0,h_1 \}\} + u_1 \{ h_1,\{h_0,h_1 \}\} + u_2 \{ h_2,\{h_0,h_1 \}\} = 0, \\
\ddot{h}_2 = \{ h_0,\{h_0,h_2 \}\}+ u_1 \{ h_1,\{h_0,h_2 \}\} + u_2 \{ h_2,\{h_0,h_2 \}\} = 0.
\end{align*}
If 
\begin{equation*}
\det \Delta_1 = \det
    \begin{pmatrix}
       \{ h_1,\{h_0,h_1 \}\}    & \{ h_2,\{h_0,h_1 \}\}   \\
       \{ h_1,\{h_0,h_2 \}\}    & \{ h_2,\{h_0,h_2 \}\}
    \end{pmatrix}
\neq 0 
\end{equation*}
along $I$, then
\begin{equation} \label{eq_uso1}
\begin{cases}
u_1 = \big( -\{ h_0,\{h_0,h_1 \}\}  \{ h_2,\{h_0,h_2 \}\} + \{ h_0,\{h_0,h_2 \}\} \{ h_2,\{h_0,h_1 \}\} \big) / \det \Delta_1, \\
u_2 = \big(   \{ h_0,\{h_0,h_1 \}\} \{ h_1,\{h_0,h_2 \}\} - \{ h_0,\{h_0,h_2 \}\} \{ h_1,\{h_0,h_1 \}\} \big) / \det \Delta_1,
\end{cases}
\end{equation}
and we say that the control $u=(u_1,u_2)$ is of \emph{order $1$} (also called \emph{minimal order} in \cite{BonnardChyba,Chitour}).
Note that $u_1$ and $u_2$ must moreover satisfy the constraint $u_1^2+u_2^2 \leq 1$. 
Note also that, if moreover $[g_1,[f,g_2]]=0$ and $[g_2,[f,g_1]]=0$, then \eqref{eq_uso1} yields
\begin{equation*} 
u_1 =  -\{ h_0,\{h_0,h_1 \}\}   / \{ h_1,\{h_0,h_1 \}\}, \quad
u_2 =  - \{ h_0,\{h_0,h_2 \}\}   /\{ h_2,\{h_0,h_2 \}\}.
\end{equation*}

Now, if $\{ h_1,\{h_0,h_1 \}\} = 0$ and $\{ h_2,\{h_0,h_2 \}\}=0$ along $I$, then we must have $\{ h_i,\{h_0,h_j \}\} = 0$, $i,j=1,2$, $i\neq j$ according to the Goh condition (see \cite{Goh, Krener}, see also below), and hence we go on differentiating. Assuming that $[ g_1,[f,g_1 ]] = 0$ and $[ g_2,[f,g_2 ]] = 0$, we have
$$
[g_i, \mathrm{ad}^2 f.g_i]] = [g_i, [f, \mathrm{ad}f.g_i]] = -[f,[\mathrm{ad}f.g_i,g_i]] - [\mathrm{ad}f.g_i, [g_i,f]] = 0, \quad i=1,2 ,
$$
and we get
\begin{equation} \label{eq_uso2}
h_1^{(3)} = \{ h_0,\mathrm{ad}^2 h_0. h_1 \} + u_2 \{ h_2,\mathrm{ad}^2 h_0. h_1 \}  = 0,\quad
h_2^{(3)} = \{ h_0,\mathrm{ad}^2 h_0. h_2 \} + u_1 \{ h_1,\mathrm{ad}^2 h_0. h_2 \}  = 0.
\end{equation}
Due to higher-order necessary conditions for optimality (see below), an optimal singular control cannot appear in a nontrivial way with an odd number of derivatives, therefore we must have $\{ h_2,\mathrm{ad}^2 h_0. h_1 \}=0 $ and $\{ h_1,\mathrm{ad}^2 h_0. h_2 \} = 0$ along $I$. Accordingly, $h_i^{(3)}=0$, $i=1,2$, gives the three additional constraints along the singular arc $\{ h_0,\mathrm{ad}^2 h_0. h_1 \}=0$, $\{ h_0,\mathrm{ad}^2 h_0. h_2 \}=0$, and $\{ h_2,\mathrm{ad}^2 h_0. h_1 \}  = -\{ h_1,\mathrm{ad}^2 h_0. h_2 \}=0$.
Derivating these constraints with respect to $t$, we get
\begin{align*}
h_1^{(4)}=\mathrm{ad}^4h_0.h_1 + u_1 \{ h_1,\mathrm{ad}^3h_0.h_1\} + u_2 \{ \mathrm{ad}^2h_0.h_1,\mathrm{ad}h_0.h_2\} = 0, \\
h_2^{(4)}=\mathrm{ad}^4h_0.h_2 + u_1 \{ \mathrm{ad}^2h_0.h_2,\mathrm{ad}h_0.h_1\} + u_2 \{ h_2,\mathrm{ad}^3h_0.h_2\}= 0 .
\end{align*}
Assuming that $\{h_i,\mathrm{ad}^3h_0.h_i\}<0$, $i=1,2$ (generalized Legendre-Clebsch condition, see below) and that
\begin{equation*}
\det \Delta_2  = 
\det
    \begin{pmatrix}
       \{ h_1,\mathrm{ad}^3h_0.h_1\}  & \{ \mathrm{ad}^2h_0.h_1,\mathrm{ad}h_0.h_2\}   \\
       \{ \mathrm{ad}^2h_0.h_2,\mathrm{ad}h_0.h_1\} & \{ h_2,\mathrm{ad}^3h_0.h_2\}
    \end{pmatrix}   
\neq 0 
\end{equation*}
along $I$, the singular control is given by
\begin{equation*}
\begin{cases}
u_1 = \big(-(\mathrm{ad}^4h_0.h_1)\{ h_2,\mathrm{ad}^3h_0.h_2\} + 
            (\mathrm{ad}^4h_0.h_2)\{ h_2,\mathrm{ad}^3h_0.h_1\} \big) / \det \Delta_2, \\
u_2 = \big( (\mathrm{ad}^4h_0.h_1) \{ h_1,\mathrm{ad}^3h_0.h_2\} -
            (\mathrm{ad}^4h_0.h_2) \{ h_1,\mathrm{ad}^3h_0.h_1\} \big) / \det \Delta_2.
\end{cases}
\end{equation*}
We say, then, that the singular control $u=(u_1,u_2)$ is of \emph{intrinsic order two}.

\paragraph{Precise definitions.}
Now, following \cite{Gabasov}, let us give a precise definition of the order of a singular control.

\begin{definition}
The singular control $u=(u_1,u_2)$ defined on a subinterval $I \subset [0,t_f]$ is said to be of \emph{order $q$} if
$$
\frac{\partial}{\partial u_i} \frac{d^{k}}{dt^{k}}(h_i) = 0,\quad k=0,1,\cdots,2q-1,
$$
$$
\frac{\partial}{\partial u_i} \frac{d^{2q}}{dt^{2q}}(h_i) \neq 0, \quad 
\det \left( \frac{\partial}{\partial u} \frac{d^{2q}}{dt^{2q}}\Phi \right) \neq 0,\quad 
i=1,2,
$$
along $I$. The control $u$ is said to be of \emph{intrinsic order} $q$ if, moreover, the vector fields satisfy 
$$
[g_i, \mathrm{ad}^k f.g_i] \equiv 0,\quad k=1,\cdots,2q-2,\quad i=1,2.
$$
\end{definition}

The condition of a nonzero determinant guarantees that the optimal control can be computed from the $2q$-th time derivative of the switching function. Note that, in the definition, it is required that the two components of the control have the same order. 

We next recall the Goh and generalized Legendre-Clebsch conditions (see \cite{Goh,Kelley,Krener}). It is worth noting that in \cite{Krener}, the following higher-order necessary conditions are given even when the components of the control $u$ have different orders.

\begin{lemma}\label{necconds0}
\emph{(higher-order necessary conditions)}
Assume that a singular control $u=(u_1,u_2)$ defined on $I$ is of order $q$ and is optimal. Then the Goh condition
$$
\frac{\partial}{\partial u_j} \frac{d^{k}}{dt^{k}}(h_i) = 0,\quad k=0,1,\cdots,2q-1,\quad i,j=1,2,\quad i \neq j,
$$
must be satisfied along $I$. Moreover, the matrix of which the $(i,j)$-th component is
$$
(-1)^q \frac{\partial}{\partial u_j} \frac{d^{2q}}{dt^{2q}}(h_i) ,\quad i,j=1,2,
$$
is symmetric and nonpositive along $I$ (generalized Legendre-Clebsch Condition).
\end{lemma}

In the problem $\MTCP$, as we will see, it happens that singular controls are of intrinsic order $2$, and that $ [ g_1,g_2 ] = 0$, $[ g_1,[f,g_2 ]] = 0$, and $[ g_2,[f,g_1 ]] = 0$,
so that the conditions given in the above definition yield $[ g_1,[f,g_1 ]] = 0$, $[ g_2,[f,g_2 ]] = 0$, $[ g_1,\mathrm{ad}^2 f.g_1 ] = 0$, $[ g_2,\mathrm{ad}^2 f.g_2] = 0$,
$\langle p,  [ g_1,\mathrm{ad}^3 f.g_1 ](x) \rangle \neq 0$, $\langle p,  [ g_2,\mathrm{ad}^3 f.g_2](x) \rangle \neq 0$, and
$$
\langle p,  [ g_1,\mathrm{ad}^3 f.g_1 ](x) \rangle \langle p,  [ g_2,\mathrm{ad}^3 f.g_2 ](x) \rangle 
- \langle p,  [ g_2,\mathrm{ad}^3 f.g_1 ](x) \rangle \langle p,  [ g_1,\mathrm{ad}^3 f.g_2 ](x) \rangle \neq 0,
$$
and we have the following higher-order necessary conditions, that will be used in the study of the problem $\MTCP$.

\begin{corollary} \label{necconds}
We assume that the optimal trajectory $x(\cdot)$ contains a singular arc, defined on the subinterval $I$ of $[0,t_f]$, associated with a control $u=(u_1,u_2)$ of intrinsic order $2$. If the vector fields satisfy $[g_1,g_2] = 0$, $[g_i,[f,g_j]]  = 0$, for $i,j=1,2$, then the Goh condition
$$
\langle p(t), [g_1,\mathrm{ad} f.g_2] (x(t)) \rangle = 0, \quad
\langle p(t), [g_1,\mathrm{ad}^2 f.g_2] (x(t)) \rangle = \langle p(t), [g_2,\mathrm{ad}^2 f.g_1] (x(t)) \rangle = 0,
$$
and the generalized Legendre-Clebsch condition (in short, GLCC )
$$
\langle p(t),[g_i,\mathrm{ad}^3f.g_i](x(t))\rangle \leq 0, \quad i=1,2,
$$
$$
\langle p(t),[g_1,\mathrm{ad}^3f.g_2](x(t))\rangle = \langle p(t),[g_2,\mathrm{ad}^3f.g_1](x(t))\rangle
$$
must be satisfied along $I$. Moreover, we say that the strengthened GLCC  is satisfied if we have a strict inequality above, that is, $\langle p(t),[g_i,\mathrm{ad}^3f.g_i](x(t))\rangle < 0$.
\end{corollary}

Corollary \ref{necconds} follows from Lemma \ref{necconds0} and from the arguments developed in the previous informal discussion. It will be used in Section \ref{ext_sin_chat}.

We next investigate the singular junctions for the problem \eqref{pb_ocp}, and the chattering phenomenon.

\subsection{Chattering phenomenon}
One can find in \cite{McDanell} some results on the junction between an optimal regular arc and an optimal singular arc, for single-control affine systems, among which a result stating that, if the singular arc is of even order and if the control is discontinuous at the junction, then the junction must be nonanalytical (meaning that the control is not piecewise analytic in any neighborhood of the junction).
In \cite{SCHATTLER,ZELIKIN}, it is proved that such a nonanalytical junction between a regular arc and a singular arc of intrinsic order two causes chattering (see also \cite{ZTC}). When the control takes values in the unit disk, explicit analytic expressions for some optimal trajectories of linear-quadratic problems were given, e.g., in \cite{Marchal,ZELIKIN}.
However these results cannot be applied to $\MTCP$ because the control system is bi-input and the cost functional is the time; they are anyway a good source of inspiration to establish the results of that section.
The following result is valid for general bi-input control-affine systems.
 
\begin{theorem} \label{theorem_chattering}
Consider the optimal control problem \eqref{pb_ocp}.
Let $(x(\cdot),p(\cdot),p^0,u(\cdot))$ be an optimal extremal lift on $[0,t_f]$. We assume that $u$ is singular of order two along an open interval $I \subset [0,t_f]$, and we denote this control by $u_s=(u_{1s},u_{2s})$. 
We assume that $\Vert u_s(t)\Vert  < 1$ (i.e., the singular control does not saturate the constraint) and that 
$\frac{\partial }{\partial u_1} \frac{d^4}{dt^4} h_2(x(t),p(t)) = 0$
along $I$. 
Then the optimal control $u$ must switch infinitely many times at the junction with the singular arc. In other words, there is a chattering phenomenon, which is due to the connection of a regular arc with a singular arc of higher-order.
\end{theorem}

\begin{proof}
Since the singular control is of order two, it follows from the definition that
$$
\frac{\partial }{\partial u_i} \frac{d^k}{dt^k} h_i(x(t),p(t)) = 0,\quad k=0,\cdots,3,\quad i=1,2, \qquad
\frac{\partial }{\partial u_i} \frac{d^4}{dt^4} h_i(x(t),p(t)) \neq 0.
$$
Thus, we get from $\frac{\partial }{\partial u_i} \frac{d^4}{dt^4} h_i(x(t),p(t)) \neq 0$ and Lemma \ref{necconds0} that 
$$
\frac{\partial }{\partial u_j} \frac{d^k}{dt^k} h_i(x(t),p(t)) = 0,\quad k=0,\cdots,3,\quad i,j=1,2,\, i\neq j,
$$
and
$$
\frac{\partial }{\partial u_i} \frac{d^4}{dt^4} h_i(x(t),p(t)) < 0, \quad 
\frac{\partial }{\partial u_1} \frac{d^4}{dt^4} h_2(x(t),p(t)) = \frac{\partial }{\partial u_2} \frac{d^4}{dt^4} h_1(x(t),p(t)),
$$
along the singular arc $I$.
By assumption, we have $\frac{\partial }{\partial u_1} \frac{d^4}{dt^4} h_2(x(t),p(t)=0$, and hence we can write $h_i^{(4)}(x(t),p(t)) = a_{i0}(x(t),p(t)) + u_{is} a_{ii}(x(t),p(t))$ with $a_{ii}(x(t),p(t)) = \frac{\partial }{\partial u_i} \frac{d^4}{dt^4} h_i(x(t),p(t)) < 0$.

Without loss of generality, we consider a concatenation of a singular arc with a regular arc at time $\tau \in I$. Assume that for some $\varepsilon > 0$ the control u is singular along $(-\varepsilon+\tau,\tau)$, and that, along $(\tau,\tau+\varepsilon)$, the control $u=(u_1,u_2)$ is given by $u_i= h_i/\Vert  \Phi \Vert  \geq 0$, $i=1,2 $. It can be easily seen from the assumption that $\Vert u_s\Vert  < 1$ that there exists at least one component of the singular control that is smaller than the same component of the regular control, i.e., $u_{ks} < u_k$ for $k=1$ or $k=2$.
Then, it follows that
\begin{equation} \label{hk4}
\begin{split}
h_k^{(4)} (\tau)
    = & \ a_{k0}(x(\tau),p(\tau)) + u_{k}(\tau) a_{kk}(x(\tau),p(\tau)) \\
\leq & \ a_{k0}(x(\tau),p(\tau)) + u_{ks}(\tau^-) a_{kk}(x(\tau),p(\tau)) 
          = h_k^{(4)} (\tau^-) = 0 .
\end{split}
\end{equation}
Hence the switching function $h_k$ has a local maximum at $t=\tau$ and is nonpositive along the interval $(\tau,\tau+\varepsilon)$. It follows from the maximization property of the Hamiltonian that $u_k \leq 0$. This is a contradiction. 
If, instead, we assume $u_i= h_i/\Vert  \Phi \Vert  \leq 0$, $i=1,2 $ over $(\tau,\tau+\varepsilon)$, then there must exists a control component $u_{ks}$ that is larger than $u_{ks}$, i.e., $u_{ks} > u_k$, and then we obtain $h_k^{(4)} (\tau) \geq h_k^{(4)} (\tau^-) =0$, which yields $u_k \geq 0$ and thus a contradiction.
Then, if we assume $u_i= h_i/\Vert  \Phi \Vert  < 0$ and $u_j= h_j/\Vert  \Phi \Vert  > 0$, $i,j=1,2$, $i \neq j$, we will have either $u_{is} < u_i$ which gives a contradiction with the sign of $u_i$, or $u_{is} \geq u_i$ and $u_{js} < u_j$ which gives a contradiction with the sign of $u_j$.
A similar reasoning can be done for regular-singular type concatenations. 

Recall that the extremal is said singular if $\Vert  \Phi(t) \Vert  = \sqrt{h_1^2(t)+h_2^2(t)} = 0$, $t\in I$. Thus, the obtained contradiction indicates that the concatenation of a singular arc with a regular arc violates the PMP and thus there exists a chattering arc when trying to connect a regular arc with a singular arc. 
\end{proof}

\begin{remark}
Note that, in this result, we have assumed that $\Vert  u_s\Vert < 1$. 
In the (nongeneric) case where the singular control saturates the constraint, in order to get the same result we need to assume that the strengthened GLCC  is satisfied at the junction point, i.e., $a_{ii}(x(\tau),p(\tau)) < 0$, and the control is discontinuous at the singular junction.

In addition, we have assumed that $\frac{\partial }{\partial u_1} \frac{d^4}{dt^4} h_2(x(t),p(t)) = 0$. Actually, if $\frac{\partial }{\partial u_1} \frac{d^4}{dt^4} h_2(x(t),p(t)) \neq 0$, then singular and regular extremals can be connected without chattering. For example, \eqref{hk4} gives
\begin{equation*}
\begin{split}
h_k^{(4)} (\tau)
    = &\ a_{k0}(x(\tau),p(\tau)) + u_{k}(\tau) a_{kk}(x(\tau),p(\tau)) +  u_{m} (\tau) a_{km}(x(\tau),p(\tau))\\
\leq &\ a_{k0}(x(\tau),p(\tau)) + u_{ks}(\tau^-) a_{kk}(x(\tau),p(\tau)) +  u_{m} (\tau) a_{km}(x(\tau),p(\tau)) \\
        &  = h_k^{(4)} (\tau^-) + a_{km}(x(\tau),p(\tau)) (u_{m}(\tau)  - u_{ms}(\tau) )= a_{km}(x(\tau),p(\tau)) (u_{m}(\tau)  - u_{ms}(\tau) ) .
 \end{split}
\end{equation*}
where $a_{km}(x(\tau),p(\tau)) =  \frac{\partial }{\partial u_m} \frac{d^4}{dt^4} h_k(x(t),p(t)) $, $k,m=1,2$, $k \neq m$. In contrast to the previous reasoning, now the fact that $a_{km}(x(\tau),p(\tau)) (u_{m} - u_{ms}) > 0$ does not raise any more a contradiction.
\end{remark}

In the next section, we analyze the regular, singular and chattering extremals for the problem $\MTCP$ by using the results presented previously.

\section{Geometric analysis of the extremals of $\MTCP$}\label{Chp_Gae}
In this section, we classify the switching points by their contact with the switching surface, and we establish that the optimal singular arcs of the $\MTCP$, if they exist, cause chattering.

\subsection{Regular extremals}\label{C_gare}
\paragraph{Normal extremals.}
Here, we consider normal extremals and we take $p^0 = -1$.
Let us consider the system \eqref{sys_multi_affine}, with the vector fields $f$, $g_1$ and $g_2$ defined by \eqref{fetgi}.
Denoting the adjoint vector by $p = (p_{v_x}, p_{v_y}, p_{v_z}, p_{\theta}, p_{\psi}, p_{\phi}, p_{\omega_x}, p_{\omega_y})$, the adjoint equations given by the PMP are
\begin{equation} \label{sys_adjoint}
\begin{split}
\dot{p}_{v_x}&= 0,\quad \dot{p}_{v_y}= 0,\quad \dot{p}_{v_z}= 0,\\
\dot{p}_{\theta}&= -a \cos \psi (p_{v_x} \cos \theta -p_{v_z} \sin \theta),\\
\dot{p}_{\psi}&= a \sin \psi \sin \theta p_{v_x} + a \cos \psi p_{v_y} + a \cos \theta \sin \psi p_{v_z}
-\sin \psi (\omega_x \sin \phi + \omega_y \cos \phi)/\cos^2 \psi p_{\theta} \\
&\qquad\qquad\qquad  -(\omega_x \sin \phi + \omega_y \cos \phi)/ \cos^2 \psi p_{\phi}, \\
\dot{p}_{\phi}& = -(\omega_x \cos \phi -\omega_y \sin \phi)/\cos \psi p_{\theta}
+(\omega_x \sin \phi + \omega_y \cos \phi) p_{\psi} \\
&\qquad\qquad\qquad -\tan \psi (\omega_x \cos \phi -\omega_y \sin \phi) p_{\phi}, \\
\dot{p}_{\omega_x}&= - \sin \phi / \cos \psi p_{\theta}- \cos \phi p_{\psi}
 -\sin \psi \sin \phi / \cos \psi p_{\phi}, \\
\dot{p}_{\omega_y}&=  - \cos \phi / \cos \psi p_{\theta}+ \sin \phi p_{\psi}
 -\sin \psi \cos \phi / \cos \psi p_{\phi} ,
\end{split}
\end{equation}
with the transversality condition
\begin{equation}\label{trc}
    p_{v_x}(t_f) \sin \theta_f \cos \psi_f - p_{v_y}(t_f) \sin \psi_f + p_{v_z}(t_f) \cos \theta_f \cos \psi_f =0 .
\end{equation}
The switching function is $\Phi(t) = (h_1(t),h_2(t)) = ( \bar{b} p_{\omega_y}(t),- \bar{b} p_{\omega_x}(t))$ and is of class $C^1$. The switching manifold $\Gamma$ is the submanifold of $\R^{16}$ of codimension two defined by
$$
\Gamma = \{ z=(x,p)\in\R^{16} \mid p_{\omega_x} = p_{\omega_y} = 0 \}.
$$
Let us fix an arbitrary reference regular extremal $z(\cdot)=(x(\cdot),p(\cdot))$ of the problem $\MTCP$.

If $z(\cdot)$ never meets $\Gamma$, then the extremal control is given by \eqref{eq_ut} along the whole extremal.

If it meets $\Gamma$ then there is a singularity to be analyzed. It is not even clear if the extremal flow is well defined when crossing such a point (we could lose uniqueness). 
Let us assume that the extremal $z(\cdot)$ meets $\Gamma$ at some time $t_0$, and we set $z_0=z(t_0)=(x_0,p_0)$.
Following \cite{BonnardCaillau, BonnardTrelat, Kupka}, we classify the regular extremals by their contact with the switching surface, i.e., if $\Phi^{(k-1)}(t_0) = 0$ and $ \Phi^{(k)}(t_0) \neq 0$ for some $k\in\N^*$, then the point $z_0$ is said to be a \emph{point of order $k$}. Without loss of generality, we assume that $t_0=0$.

Let us analyze the singularity occuring at points of order $1$, $2$, $3$ and $4$ for the problem $\MTCP$.

\paragraph{Points of order $1$.} 
\begin{lemma} \label{lem_pointorder1}
We assume that $z_0$ is of order $1$. Then the reference extremal is well defined in a neighborhood of $t=0$, in the sense that there exists a unique extremal associated with the control $u=(u_1,u_2)$ passing through the point $z_0$. The control turns with an angle $\pi$ when passing through the switching surface $\Gamma$, and is locally given by
\begin{equation*} 
u_1(t) = \frac{a_1}{\sqrt{a_1^2+a_2^2}} \frac{t}{|t|} +\mathrm{o}(1) ,\qquad
u_2(t) = \frac{a_2}{\sqrt{a_1^2+a_2^2}} \frac{t}{|t|} +\mathrm{o}(1) ,
\end{equation*}
with $a_1 = \{h_0,h_1\}(z_0)$, $a_{2}= \{h_0,h_2\}(z_0)$.
\end{lemma}

\begin{proof}
In the problem $\MTCP$, the vector fields $g_1$ and $g_2$ (defined by \eqref{fetgi}) commute, i.e., $[g_1,g_2] = 0$. By derivating the switching function, we get
$\dot{h}_1=\{ h_0,h_1 \} $ and $\dot{h}_2 = \{ h_0,h_2 \} $ along the extremal, and then, since the point is of order $1$, we have, locally, $h_1(t) = a_1 t + \mathrm{o}(t)$ and $h_2(t) = a_2 t +\mathrm{o}(t)$, with $a_1^2+a_2^2 \neq 0$, and we also have 
$p_{\omega_x}(t)=-a_2 t/\bar{b} + \mathrm{o}(t)$, $p_{\omega_y}(t)=a_1 t/\bar{b} + \mathrm{o}(t)$, $\sqrt{p_{\omega_x}(t)^2+p_{\omega_y}(t)^2}=\sqrt{a_1^2+a_2^2} |t|/\bar{b} +\mathrm{o}(t)$. The expression of the optimal control in the lemma follows.
At the crossing point, both control components change their sign, i.e., $u_i(0^-)=-u_i(0^+)$, $i=1,2$, which means that the control direction turns with an angle $\pi$ when crossing $\Gamma$.
\end{proof}

\paragraph{Points of order 2.} 

\begin{lemma} \label{lem_points2}
We assume that $z_0$ is of order $2$. Then the reference extremal is well defined in a neighborhood of $t=0$, in the sense that there exists a unique extremal associated with the control $u=(u_1,u_2)$ passing through the point $z_0$. Moreover, the switching function is of class $C^\infty$ in the neighborhood of $t=0$, the control is of class $C^\infty$, and we have
\begin{equation*}
u_1(t) = \frac{\alpha_1}{\sqrt{\alpha_1^2 + \alpha_2^2}} + \mathrm{o}(1),\qquad 
u_2(t) =\frac{\alpha_2}{\sqrt{\alpha_1^2 + \alpha_2^2}} +\mathrm{o}(1) ,
\end{equation*}
with $\alpha_1 = \{ h_0, \{ h_0,h_1\}\}(z_0)$ and $\alpha_2 = \{ h_0, \{ h_0,h_2\}\}(z_0)$.
\end{lemma}

\begin{proof}
The vector fields $f$, $g_1$ and $g_2$, defined by \eqref{fetgi}, are such that $[g_i,[f,g_j]] = 0$, for $i,j=1,2$ (see Lemma \ref{Lieconfig}). Then, according to the calculations done in Section \ref{sec32}, we have $
\ddot{h}_1 = \{ h_0, \{ h_0,h_1\}\}$ and $\ddot{h}_2 = \{ h_0, \{ h_0,h_2\}\}$, and hence the functions  $t \mapsto h_i(t)$, $i=1,2$ are of class $C^2$ at $0$. Locally, we have $h_1(t)=\frac{1}{2}\alpha_1 t^2+\mathrm{o}(t^2)$ and $h_2(t)=\frac{1}{2}\alpha_2 t^2+\mathrm{o}(t^2)$. The expression of the optimal control follows and the control is continuous.

Differentiating again the switching function, we have $h_1^{(3)}(t) = \mathrm{ad}^3 h_0.h_1 + u_2 \{h_2, \mathrm{ad}^2 h_0.h_1\}$ and $h_2^{(3)}(t) = \mathrm{ad}^3 h_0.h_2 + u_1 \{h_1, \mathrm{ad}^2 h_0.h_2\}$, because $[g_i,[f,[f,g_i]]] = 0$, for $i=1,2$  (see Lemma \ref{Lieconfig}). Since the control is continuous, the switching function is at least of class $C^3$ at $0$. 
Hence, locally we can write $h_1(t)=\frac{1}{2}\alpha_1 t^2+ \frac{1}{6}\beta_1 t^3 +\mathrm{o}(t^3)$ and $h_2(t)=\frac{1}{2}\alpha_2 t^2+ \frac{1}{6}\beta_2 t^3 +\mathrm{o}(t^3)$, where $\beta_i = \mathrm{ad}^3 h_0.h_i(z_0) +  \alpha_j/\sqrt{\alpha_1^2+\alpha_2^2}  \{h_j, \mathrm{ad}^2 h_0.h_i\}(z_0)$ for $i,j=1,2$ and $i \neq j$. We infer that the control is at least of class $C^1$ at $0$, with $\dot{u}_1(0) = \frac{1}{6} (\alpha_2^2 \beta_1 - 2 \alpha_1\alpha_2 \beta_2)/(\alpha_1^2+\alpha_1^2)^{3/2}$ and $\dot{u}_2(0) = \frac{1}{6} (\alpha_1^2 \beta_2 - 2 \alpha_1\alpha_2 \beta_1)/(\alpha_1^2+\alpha_1^2)^{3/2}$.
We get the smoothness by an immediate induction argument: for $k>2$, assuming that the switching function is of class $C^k$ and that the control is of class $C^{k-2}$, then the $(k+1)$-th time derivative of the switching function can be written as $h_i^{(k+1)} = u_i^{(k-2)} \mathrm{term}_1 + \mathrm{term}_2$,
where $\mathrm{term}_i$, $i=1,2$ are terms involving time derivatives of the control of order lower than $k-2$. Hence the switching function is of class $C^{(k+1)}$ since $u$ is of class $C^{k-2}$, and $h_i = \sum_{p=2}^{k+1} a_{i,p} t^p + \mathrm{o}(t^{k+1})$ where all coefficients can be computed explicitly. The $(k-1)$-th time derivative of the control can be computed using the coefficients $a_{i,p}$, $p=2,\cdots,k+1$, and hence the control is of class $C^{k-1}$. The result follows.
\end{proof}

\paragraph{Points of order 3.}
\begin{lemma} \label{lem_pointorder3}
We assume that $z_0$ is of order $3$. Then the reference extremal is well defined in a neighborhood of $t=0$, in the sense that there exists a unique extremal associated with the control $u=(u_1,u_2)$ passing through the point $z_0$. 
If $b_i = \mathrm{ad}^3h_0.h_i (z_0) \neq 0$ and $c = \{h_2,\mathrm{ad}^2h_0.h_1\}(z_0) \neq 0$, then
the switching function is of class $C^2$ in the neighborhood of $t=0$ and the control is discontinuous when passing through the switching surface $\Gamma$ locally, and we have
$$
u_i (t) = \frac{\beta_i^-}{\sqrt{\beta_1^{-2}+\beta_2^{-2}}} + \mathrm{o}(1),\quad t<0, \qquad
u_i (t) = \frac{\beta_i^+}{\sqrt{\beta_1^{+2}+\beta_2^{+2}}} + \mathrm{o}(1),\quad t>0,
$$
where, setting $d=\sqrt{(-c^2+b_1^2+b_2^2)}$ and $e_{ij}=- b_i c^2 + b_i b_j^2 + b_i^3$,
$$
\beta_i^- =  \frac{ - (-1)^{j} b_j c d +e_{ij} } {b_1^2+b_2^2},\qquad
\beta_i^+ = \frac{(-1)^{j} b_j c d +e_{ij}} {b_1^2+b_2^2},\qquad i,j=1,2,\quad i\neq j.
$$
\end{lemma}

\begin{proof}
Using \eqref{fetgi} and \eqref{eq_uso2}, we have 
$h_1^{(3)} = \mathrm{ad}^3h_0.h_1 + u_2 \{h_2,\mathrm{ad}^2h_0.h_1\}$ and $h_2^{(3)} = \mathrm{ad}^3h_0.h_2- u_1 \{h_2,\mathrm{ad}^2h_0.h_1\}$ (see also Lemma \ref{Lieconfig}). 
Assuming that we have locally $h_i(t) = \frac{1}{6} \beta_i t^3 + o(t^3)$, we infer that $u_i (t) = \frac{\beta_i}{\sqrt{\beta_1^{2}+\beta_2^{2}}}\frac{t}{|t|} + o(1)$. By substituting the control into the expression of $h_i^{(3)} $, we get
$\beta_1 = b_1 + c  \frac{\beta_2}{\sqrt{\beta_1^{2}+\beta_2^{2}}}$ and $\beta_2 = b_2 - c  \frac{\beta_1}{\sqrt{\beta_1^{2}+\beta_2^{2}}}$. The result follows by solving for $t>0$ and $t<0$.
\end{proof}

\begin{remark}
If $c = \{h_2,\mathrm{ad}^2h_0.h_1\} (z_0)= 0$, then the switching function is of class $C^3$ at $0$ and the control turns with an angle $\pi$ when passing through the switching surface $\Gamma$.
\end{remark}

\paragraph{Points of order $4$.}
We assume that $z_0$ is of order $4$.
If $ \{h_2,\mathrm{ad}^2h_0.h_1\}(z_0) \neq 0$, then $0 = h_1^{(3)} = b_1 + u_2 c$ and $0 = h_2^{(3)} = b_2 - u_1 c$, where $b_i = \mathrm{ad}^3h_0.h_i (z_0)$, $c=\{h_2,\mathrm{ad}^2h_0.h_1\}(z_0)$.
We have $u_1 = - b_1 / c$ and $u_2 = b_2 /c$.  If moreover 
$b_1^2 +  b_2^2 = c^2$, which indicates that the control $u_i=\alpha_i/\sqrt{\alpha_1^2+\alpha_2^2}$ can be a regular control according to the value of $h_i(t)=\frac{1}{8}\alpha_i t^4 +o(t^4)$, $i=1,2$, at time $0$, we get that $u_2/u_1=\alpha_2/\alpha_1=-b_1/b_2$, $\mathrm{sign} (\alpha_1)=b_2/c $ and $\mathrm{sign} (\alpha_2) = -b_1/c $. Then
$$
u_1(t) = \mathrm{sign} (c) \frac{b_2}{\sqrt{b_1^2+b_2^2}}+ \mathrm{o}(1),\quad
u_2(t) = -\mathrm{sign} (c) \frac{b_1}{\sqrt{b_1^2+b_2^2}}+ \mathrm{o}(1).
$$
If
\begin{equation*}
\begin{split}
\alpha_1 = h_1^{(4)}(z_0) = \{h_0,b_1\}(z_0) + u_1\{h_1,b_1\}(z_0)+u_2(\{h_2,b_1\}(z_0) +\{h_1,c\}(z_0) )\\
  +u_1u_2\{h_1,c\}(z_0) +u_2^2\{h_2,c\}(z_0) ,\\
\alpha_2 = h_2^{(4)}(z_0) = \{h_0,b_2\}(z_0) + u_1( \{h_1,b_2\}(z_0) - \{h_1,c\}(z_0))+u_2 \{h_2,b_2\}(z_0) \\
  - u_1^2\{h_1,c\}(z_0) - u_1u_2\{h_2,c\}(z_0) ,  
\end{split}
\end{equation*}
then the extremal is well defined in a neighborhood of $t=0$ and the control is continuous when passing the switching surface $\Gamma$.

If $ \{h_2,\mathrm{ad}^2h_0.h_1\}(z_0) = 0$, then $ \{h_0,b_1\}(z_0) = \{h_0,b_2\}(z_0) =0$, $\{h_1,b_2\}(z_0) =0$, $\{h_2,b_1\}(z_0) =0$ and $\{h_1,b_1\}(z_0)=\{h_2,b_2\}(z_0)$ (see the proof of Lemma \ref{lem_singulararcs} further), and we have
$h_1^{(4)}(z_0) = u_1\{h_1,b_1\}(z_0)$ and $h_2^{(4)}(z_0)  = u_2 \{h_1,b_1\}(z_0)$.
Assuming that we have locally $\Phi(t)=R_0 t^4 e^{i \alpha \mathrm{ln} |t|} +\mathrm{o}(t^4) = R_0 t^4(\cos(\alpha \mathrm{ln} |t|), \sin(\alpha \mathrm{ln} |t|))+\mathrm{o}(t^4)$ (we identify $\C=\R^2$ for convenience), with $R_0 >0$, we get $u(t)=\Phi(t)/\Vert \Phi(t)\Vert  = R_0 e^{i \alpha \mathrm{ln} |t|}$ and
$$
\Phi^{(4)} (t) = R_0(4+i\alpha)(3+i\alpha)(2+i\alpha)(1+i\alpha) e^{i \alpha \mathrm{ln} |t|} + \mathrm{o}(1),
$$
which leads to $R_0(\alpha^4-35\alpha^2+24) = \{h_1,b_1\}(z_0)$ and $R_0(-10\alpha^3+50\alpha) =  \{h_1,b_1\}(z_0)$.
It follows that $R_0=\{h_1,b_1\}(z_0)/(\alpha^4-35\alpha^2+24)$ with $\alpha \in\{ 1370/391,-871/614\}$ if $\{h_1,b_1\}(z_0)<0$, and $\alpha \in\{ 578/1493,-5650/453\}$ if $\{h_1,b_1\}(z_0)>0$.
It is clear that the uniqueness of the extremal when crossing the point $z_0$ does not hold true anymore. The switching function $\Phi(t)$ converges to $(0,0)$ when $t \to 0$, while the control switches infinitely many times when $t \to 0$. Indeed, we will see further that this situation is related the chattering phenomenon.

\paragraph{Abnormal extremals.}
Abnormal extremals correspond to $p^0=0$ in the PMP. We suspect the existence of optimal abnormal extremals in the problem $\MTCP$ for certain (nongeneric) terminal conditions.
In the planar version of the problem $\MTCP$ studied in \cite{ZTC}, if the optimal control switches at least two times then there is no abnormal minimizer. We expect that the same property is still true here. We are able to prove that the singular extremals of the problem $\MTCP$ are normal (see section \ref{ext_sin_chat}), however, we are not able to establish a clear relationship between the number of switchings and the existence of abnormal minimizers  as in \cite{ZTC}. 
Thus, in our numerical simulations further, we will assume that there is at least one normal extremal for problem $\MTCP$ and compute it.
Note moreover that Lemmas \ref{lem_pointorder1}, \ref{lem_points2} and \ref{lem_pointorder3} are also valid for abnormal extremals.

\subsection{Singular and chattering extremals}
\label{ext_sin_chat}
Let us compute the singular arcs of the problem $\MTCP$. 
According to \cite{BonnardChyba}, the singular trajectories are feedback invariants since they correspond to the singularities of the end-point mapping. This concept is related to the feedback group induced by the feedback transformation and the corresponding control systems are said to be feedback equivalent (see \cite[Section 4]{BonnardChyba} for details). 
Recall that, equivalently, a trajectory $x(\cdot)$ associated with a control $u$ is said to be singular if the differential of the end-point mapping is not of full rank. The end-point mapping $E: \mathbf{R}^n \times \mathbf{R}\times L^\infty(0,+\infty;\mathbf{R}) \mapsto \mathbf{R}^n$ is defined as by $E(x_0,t_f,u)=x(x_0,t_f,u)$ where $t \mapsto x(x_0,t,u)$ is the trajectory solution of the control system associated to $u$ such that $x(x_0,0,u)=x_0$.

Hence, we can replace the vector fields $g_1$ and $g_2$ with $\widetilde{g}_1 = \frac{\partial}{\partial \omega_y}$ and $\widetilde{g}_2 = \frac{\partial}{\partial \omega_x}$.
Let us make precise the Lie bracket configuration of the control system associated with the vector fields $f$, $\widetilde{g}_1$ and $\widetilde{g}_2$.

\begin{lemma}\label{Lieconfig}
We have
$$
\widetilde{g}_1 =   \frac{\partial}{\partial \omega_y},
\qquad 
\widetilde{g}_2  =   \frac{\partial}{\partial \omega_x},
\qquad 
[\widetilde{g}_1, \widetilde{g}_2 ]=0 ,
$$
$$
\mathrm{ad}f. \widetilde{g}_1  =   - \cos \phi / \cos \psi \frac{\partial}{\partial \theta} +  \sin \phi \frac{\partial}{\partial \psi} - \tan \psi \cos \phi \frac{\partial}{\partial \phi}= 0 ,
$$
$$
\mathrm{ad}f. \widetilde{g}_2 =  -  \sin \phi/ \cos \psi \frac{\partial}{\partial \theta} -  \cos \phi \frac{\partial}{\partial \psi} -  \tan \psi \sin \phi \frac{\partial}{\partial \phi}= 0 ,
$$
$$
\mathrm{ad}^2 f. \widetilde{g}_1 = - \omega_x \frac{\partial}{\partial \phi} + a (\cos \theta \cos \phi + \sin \theta \sin \phi \sin \psi) \frac{\partial}{\partial v_x}   - a (\cos \phi \sin \theta- \sin \phi \cos \theta \sin \psi) \frac{\partial}{\partial v_z}  ,
$$
\begin{multline*}
\mathrm{ad}^2 f. \widetilde{g}_2 = \omega_y\frac{\partial}{\partial \phi}  + a (\cos \theta \sin \phi - \sin \theta \cos \phi \sin \psi) \frac{\partial}{\partial v_x}   -a \cos \phi \cos \psi \frac{\partial}{\partial v_y} \\
- a (\sin \phi \sin \theta + \cos \phi \cos \theta \sin \psi) \frac{\partial}{\partial v_z} ,
\end{multline*}
$$
[\widetilde{g}_1,\mathrm{ad}f. \widetilde{g}_1] = [\widetilde{g}_1,\mathrm{ad}f. \widetilde{g}_2] = [\widetilde{g}_2,\mathrm{ad}f. \widetilde{g}_1] = [\widetilde{g}_2,\mathrm{ad}f. \widetilde{g}_2] =0, 
$$
\begin{multline*}
\mathrm{ad}^3 f. \widetilde{g}_1 = \omega_x \Omega_1 / \cos \psi \frac{\partial}{\partial \theta}
- \omega_x \Omega_2 \frac{\partial}{\partial \psi} +\omega_x \tan \psi \Omega_1 \frac{\partial}{\partial \phi} \\
- a \omega_y \cos \psi \sin \theta  \frac{\partial}{\partial v_x} +  a \omega_y \sin \psi  \frac{\partial}{\partial v_y} -  a \omega_y \cos \psi \cos \theta  \frac{\partial}{\partial v_z},
\end{multline*}
\begin{multline*}
\mathrm{ad}^3 f. \widetilde{g}_2 = - \omega_y \Omega_1 /\cos \psi  \frac{\partial}{\partial \theta}  + \omega_x \Omega_2  \frac{\partial}{\partial \psi}  -  \omega_x \tan \psi \Omega_1  \frac{\partial}{\partial \phi} \\
- a \omega_x \cos \psi \sin \theta  \frac{\partial}{\partial v_x}  + a \omega_x \sin \psi  \frac{\partial}{\partial v_y} - a \omega_x \cos \psi \cos \theta  \frac{\partial}{\partial v_z},
\end{multline*}
$$
[\widetilde{g}_1,\mathrm{ad}^2 f. \widetilde{g}_1] = [\widetilde{g}_2,\mathrm{ad}^2 f. \widetilde{g}_2]=0, \qquad  [\widetilde{g}_1,\mathrm{ad}^2 f. \widetilde{g}_2] = - [\widetilde{g}_2,\mathrm{ad}^2 f. \widetilde{g}_1]= \frac{\partial}{\partial \phi},
$$
\begin{multline*}
\mathrm{ad}^4 f. \widetilde{g}_1 = -a(\omega_x^2 + \omega_y^2)(\cos \phi \cos \theta + \sin \phi \sin \psi \sin \theta) \frac{\partial}{\partial v_x} -a\cos \psi \sin \phi(\omega_x^2 + \omega_y^2)  \frac{\partial}{\partial v_y} \\
+ a(\omega_x^2 + \omega_y^2)(\cos \phi*\sin \theta - \cos \theta \sin \phi \sin \psi) \frac{\partial}{\partial v_z} + (\omega_x^3 + \omega_x \omega_y^2)\frac{\partial}{\partial \phi} ,
\end{multline*}
\begin{multline*}
\mathrm{ad}^4 f. \widetilde{g}_2 = -a(\omega_x^2 + \omega_y^2)(\cos \theta \sin \phi - \cos \phi \sin \psi \sin \theta) \frac{\partial}{\partial v_x}+a \cos \phi \cos \psi(\omega_x^2 + \omega_y^2)\frac{\partial}{\partial v_y}  \\	
+ a(\omega_x^2 + \omega_y^2)(\sin \phi \sin \theta+ \cos \phi \cos \theta \sin \psi)\frac{\partial}{\partial v_z}  +(- \omega_x^2 \omega_y - \omega_y^3)\frac{\partial}{\partial \phi},
\end{multline*}
\begin{multline*}
[\widetilde{g}_1,\mathrm{ad}^3 f. \widetilde{g}_1] = -a \cos \psi \sin \theta \frac{\partial}{\partial v_x}+ a \sin \psi \frac{\partial}{\partial v_y} -a \cos\psi \cos \theta \frac{\partial}{\partial v_z} -(\omega_x \sin \phi)/\cos \psi \frac{\partial}{\partial \theta}\\
- \omega_x \cos \phi \frac{\partial}{\partial \psi} -\omega_x \sin \phi \tan \psi \frac{\partial}{\partial \phi},
\end{multline*}
\begin{multline*}
[\widetilde{g}_1,\mathrm{ad}^3 f. \widetilde{g}_2] = -2 a \omega_y (\cos \theta \sin \phi - \cos \phi \sin \psi \sin \theta) \frac{\partial}{\partial v_x}+ 2 a \omega_y \cos \phi \cos \psi \frac{\partial}{\partial v_y} \\
+ 2a\omega_y(\sin \phi \sin \theta + \cos \phi \cos \theta \sin \psi) \frac{\partial}{\partial v_z}+ (- \omega_x^2 - 3\omega_y^2) \frac{\partial}{\partial \phi},
\end{multline*}
\begin{multline*}
[\widetilde{g}_2,\mathrm{ad}^3 f. \widetilde{g}_1] = -2 a \omega_x(\cos \phi \cos \theta + \sin \phi \sin \psi \sin \theta)\frac{\partial}{\partial v_x}  -2 a \omega_x \cos \psi \sin \phi \frac{\partial}{\partial v_y} \\
+ 2a\omega_x(\cos \phi \sin \theta - \cos \theta \sin \phi \sin \psi)\frac{\partial}{\partial v_z}+(3\omega_x^2 + \omega_y^2) \frac{\partial}{\partial \phi},
\end{multline*}
\begin{multline*}
[\widetilde{g}_2,\mathrm{ad}^3 f. \widetilde{g}_2] = -a \cos \psi \sin \theta \frac{\partial}{\partial v_x}+a\sin \psi \frac{\partial}{\partial v_y} -a\cos \psi \cos \theta \frac{\partial}{\partial v_z} -(\omega_y \cos \phi)/\cos \psi \frac{\partial}{\partial \theta}\\
+ \omega_y \sin \phi \frac{\partial}{\partial \psi}-\omega_y \cos \phi \tan \psi \frac{\partial}{\partial \phi},
\end{multline*}
where $\Omega_1=\omega_x\cos\phi-\omega_y\sin\phi$ and $\Omega_2=\omega_x\sin\phi+\omega_y\cos\phi$.
Moreover, we have
$$
\dim \mathrm{Span}\big( \widetilde{g}_1 , \widetilde{g}_2 ,  \mathrm{ad}f.\widetilde{g}_1,  \mathrm{ad}f.\widetilde{g}_2, \mathrm{ad}^2 f.\widetilde{g}_1,\mathrm{ad}^2f.\widetilde{g}_2 \big)=6 .
$$
\end{lemma}

\begin{lemma} \label{lem_singulararcs}
In the problem $\MTCP$, let us assume that $(x(\cdot),p(\cdot),p^0,u(\cdot))$ is a singular arc along the subinterval $I$, which is locally optimal in $C^0$ topology. 
Then we have $u=(u_1,u_2)=(0,0)$ along $I$, and $u$ is a singular control of intrinsic order two. Moreover, the extremal must be normal, i.e., $p^0 \neq 0$, and the GLCC
\begin{equation}\label{lille1532}
a + g_x\sin\theta\cos\psi-g_y\sin\psi+g_z\cos\theta\cos\psi \geq 0,
\end{equation}
must hold along $I$.
\end{lemma}

\begin{proof}
Using Lemma \ref{Lieconfig}, we infer from $\Phi = 0$ and $\dot{\Phi} = 0$ that
\begin{equation}\label{Lie_1}
\begin{split}
&	\langle p, \widetilde{g}_1(x) \rangle =   p_{\omega_y} =0,\quad
	\langle p, \widetilde{g}_2(x) \rangle =   p_{\omega_x} =0,\\
&	\langle p,\mathrm{ad}f. \widetilde{g}_1(x) \rangle =   - p_{\theta} \cos \phi / \cos \psi+p_{\psi} \sin \phi -p_{\phi} \tan \psi \cos \phi = 0,\\
&	\langle p,\mathrm{ad}f. \widetilde{g}_2(x) \rangle =   - p_{\theta} \sin \phi/ \cos \psi-p_{\psi} \cos \phi -p_{\phi} \tan \psi \sin \phi = 0,
\end{split}
\end{equation}
and from $\ddot{\Phi} = 0$, that
\begin{equation}\label{Lie_2}
\begin{split}
	\langle p,\mathrm{ad}^2f.\widetilde{g}_1(x) \rangle = &    - \omega_x p_{\phi}+a (\cos \theta \cos \phi + \sin \theta \sin \phi \sin \psi) p_{v_x} \\
	    & - a (\cos \phi \sin \theta- \sin \phi \cos \theta \sin \psi) p_{v_z}  = 0, \\
	\langle p, \mathrm{ad}^2f.\widetilde{g}_2(x) \rangle = & \ \omega_y p_{\phi} +a  (\cos \theta \sin \phi - \sin \theta \cos \phi \sin \psi) p_{v_x} -a \cos \phi \cos \psi p_{v_y}\\
	    &  - a (\sin \phi \sin \theta + \cos \phi \cos \theta \sin \psi) p_{v_z} = 0, 
\end{split}
\end{equation}
along the interval $I$.
Since $\dim \mathrm{Span}\big( \widetilde{g}_1 , \widetilde{g}_2 ,  \mathrm{ad}f.\widetilde{g}_1,  \mathrm{ad}f.\widetilde{g}_2, \mathrm{ad}^2 f.\widetilde{g}_1,\mathrm{ad}^2f.\widetilde{g}_2 \big)=6$, the six equations in \eqref{Lie_1}-\eqref{Lie_2} are independent constraints along the singular arc. 
Therefore, writing $\Phi^{(3)} = 0$, we get from Theorem \ref{necconds} that
\begin{equation}\label{Lie_3}
\begin{split}
	\langle p,[\widetilde{g}_1, \mathrm{ad}^2f.\widetilde{g}_2(x)] \rangle = & p_{\phi} = 0,
	\qquad
	\langle p,[\widetilde{g}_2, \mathrm{ad}^2f.\widetilde{g}_1(x)] \rangle = -p_{\phi}  = 0,\\
	\langle p, \mathrm{ad}^3f.\widetilde{g}_1(x)\rangle= & p_{\theta} \omega_x \Omega_1 /\cos \psi 
                                             -p_{\psi}\omega_x \Omega_2
                                             +p_{\phi} \omega_x \tan \psi \Omega_1 
                                             -p_{v_x} a \omega_y \cos \psi \sin \theta \\
                                             &   
                                               +p_{v_y} a \omega_y \sin \psi 
                                               -p_{v_z} a \omega_y \cos \psi \cos \theta=0,\\
	\langle p,\mathrm{ad}^3f.\widetilde{g}_2(x)\rangle= & -p_{\theta} \omega_y \Omega_1 /\cos \psi 
                                            +p_{\psi}\omega_x \Omega_2
                                            -p_{\phi} \omega_x \tan \psi \Omega_1 
                                            -p_{v_x} a \omega_x \cos \psi \sin \theta\\
                                            & 
                                              +p_{v_y} a \omega_x \sin \psi 
                                              -p_{v_z} a \omega_x \cos \psi \cos \theta=0,\\
\end{split}
\end{equation}
These four constraints are dependent: they reduce to two functionally independent constraints. Hence, with \eqref{Lie_1}-\eqref{Lie_2}, we have $8$ independent constraints along $I$.
Now, we infer from \eqref{Lie_1}-\eqref{Lie_2}-\eqref{Lie_3} that $p_{\omega_x}=p_{\omega_y}=0$ and $p_{\theta}=p_{\psi}=p_{\phi}=0$ along $I$. Derivating $p_{\theta}=0$ and $p_{\psi}=0$, we get
\begin{equation} \label{pvx}
p_{v_x} = \tan \theta p_{v_z},\quad p_{v_y}=-\tan \psi / \cos \theta p_{v_z},
\end{equation}
and derivating again, that $\dot{\theta}=\dot{\psi}=0$. It follows that $\omega_x=\omega_y=0$. 
Using that $H=0$ along any extremal, we get
\begin{equation} \label{pvzpvy}
p_{v_z}= \frac{-p^0 \cos\theta\cos\psi}{a+g_x\sin \theta \cos \psi-g_y\sin \psi+g_z\cos \theta \cos \psi},
\end{equation}
Substituting \eqref{pvx} and \eqref{pvzpvy} into $\langle p,\mathrm{ad}^4f.\widetilde{g}_1\rangle$ and $\langle p,\mathrm{ad}^4f.\widetilde{g}_2\rangle$, we get 
\begin{equation*}\label{}
\begin{split}
	& \langle p,\mathrm{ad}^4f.\widetilde{g}_1\rangle =  \langle p,\mathrm{ad}^4f.\widetilde{g}_2\rangle = 0,
	\qquad 
	\langle p,[ \widetilde{g}_1, \mathrm{ad}^3f.\widetilde{g}_2 ](x)\rangle = \langle p,[ \widetilde{g}_2, \mathrm{ad}^3f.\widetilde{g}_1 ](x)\rangle = 0,\\
	& \langle p,[ \widetilde{g}_1, \mathrm{ad}^3f.\widetilde{g}_1 ](x)\rangle = \langle p,[ \widetilde{g}_2, \mathrm{ad}^3f.\widetilde{g}_2 ](x)\rangle = -\frac{a p_{v_z}}{\cos \psi \cos \theta}.
\end{split}
\end{equation*}
To prove that $u$ is of intrinsic order two, it suffices to prove that $\langle p,[ \widetilde{g}_i, \mathrm{ad}^3f.\widetilde{g}_i ](x)\rangle \neq 0$ along $I$. We prove it by contradiction. If $\langle p,[ \widetilde{g}_i, \mathrm{ad}^3f.\widetilde{g}_i ](x)\rangle = 0$, then necessarily $p_{v_z}=0$ and this would lead to $p_{v_x}=p_{v_y}=0$. It follows then from $H=0$ that $p^0=0$. We have obtained that $(p,p^0)=0$, which is a contradiction. 

The fact that $u=(u_1,u_2)=(0,0)$ simply follows from the fact that
\begin{equation*}
u_1 = - \{ h_1,\mathrm{ad}^3h_0.h_1\} /  \mathrm{ad}^4h_0.h_1, \qquad
u_2 = - \{ h_2,\mathrm{ad}^3h_0.h_2\} / \mathrm{ad}^4h_0.h_2.
\end{equation*}

Besides, if $p^0=0$, then $p_{v_z}=0$ and $p_{v_x}=p_{v_y}=0$, which leads to $(p,p^0)=0$ and thus to a contradiction as well. Therefore, $p^0 <0$ (i.e., the singular arc is normal), and then \eqref{lille1532} follows by applying the GLCC of Corollary \ref{necconds}.
\end{proof}

We define the singular surface $S$, which is filled by singular extremals of the problem $\MTCP$, by
\begin{multline} \label{ssurface}
S = \Big\{(x,p)\ \mid\ \omega_x=\omega_y=0, \quad p_{\theta}=p_{\psi}=p_{\phi}=p_{\omega_x}=p_{\omega_y}=0,\quad p_{v_x} = \tan \theta p_{v_z}, \\
 p_{v_z}= \frac{-p^0 \cos\theta\cos\psi}{a+g_x\sin \theta \cos \psi-g_y\sin \psi+g_z\cos \theta \cos \psi}, \quad p_{v_y}=-\tan \psi / \cos \theta p_{v_z} \Big\}.
\end{multline}
We will see, in the next section, that the solutions of the problem of order zero (defined in Section \ref{pb_auxi}) live in this singular surface $S$.

The following result, establishing chattering for the problem $\MTCP$, is a consequence of Theorem \ref{theorem_chattering}, Lemma \ref{Lieconfig} and Lemma \ref{lem_singulararcs}.

\begin{corollary} \label{cor_chattering}
For the problem $\MTCP$, any optimal singular arc cannot be connected with a nontrivial bang arc. There is a chattering arc when trying to connect a regular arc with an optimal singular arc. More precisely, let $u$ be an optimal control, solution of $\MTCP$, and assume that $u$ is singular on the sub-interval $(t_1,t_2)\subset[0,t_f]$ and is regular elsewhere. If $t_1>0$ (resp., if $t_2<t_f$) then, for every $\varepsilon>0$, the control $u$ switches an infinite number of times over the time interval $[t_1-\varepsilon,t_1]$ (resp., on $[t_2,t_2+\varepsilon]$).
\end{corollary}

This result is important for solving the problem $\MTCP$ in practice. Indeed, when using numerical methods to solve the problem, the chattering control is an obstacle to convergence, especially when using an indirect approach (shooting). The existence of the chattering phenomenon in the problem $\MTCP$ explains well why the indirect methods may fail for certain terminal conditions.

Note that, in the planar version of the problem $\MTCP$ studied in \cite{ZTC}), one can give sufficient conditions on the initial conditions under which the chattering phenomenon does not occur. 
Unfortunately, we are not able to derive such conditions in the general problem $\MTCP$.

\section{Numerical approaches}\label{Chp_continuation}
In this section, we design two different numerical strategies for solving the problem $\MTCP$: one is based on combining indirect methods with numerical continuation, and the other is based on a direct transcription approach. The first one may be successfully implemented when dealing with solutions without chattering arcs, and the second one is more appropriate to compute solutions involving chattering arcs.
However, both approaches are difficult to initialize successfully because the problem $\MTCP$ is of quite high dimension, is highly nonlinear, and moreover, as a main reason, the system consists of fast (Euler angles and angular velocity) and slow (orbit velocity) dynamics at the same time. 

The occurrence of chattering arcs is an obstacle to convergence. Especially for indirect methods, the chattering phenomenon raises an important difficulty due to the numerical integration of the discontinuous Hamiltonian system. Direct transcription approaches provide a sub-optimal solution of the problem that has a finite number of switchings based on a (possibly rough) discretization.
Actually, in case of chattering, we are also able to provide a sub-optimal solution with our indirect approach, by stopping the continuation before it would fail due to chattering. Though the sub-optimal solutions provided in this way may be ``less optimal" compared with those given by a direct approach, in practice they can be computed in a much faster way and also much more accurately.

\subsection{Indirect method and numerical continuation}\label{C_ic}
The idea of this continuation procedure is to use the (easily computable) solution of a simpler problem, that we call herefter the \emph{problem of order zero}, in order then to initialize an indirect method for the more complicated problem $\MTCP$. 
Then we are going to plug this simple, low-dimensional problem in higher dimension, and then come back to the initial problem by using appropriate continuations.

This method actually gives an optimal solution with high accuracy. The problem of order zero defined below is used as the starting problem because the orbit movement is much slower compared with the attitude movement and it is easy to solve explicitly. As well, it is worth noting that the solution of the problem of order zero is contained in the singular surface $S$ filled by the singular solutions for the problem $\MTCP$, defined by \eqref{ssurface}.

\subsubsection{Two auxiliary problems}
\label{pb_auxi}
\paragraph{Problem of order zero.}
We define the \emph{problem of order zero}, denoted by $\OCPZ$, as a ``subproblem'' of the complete problem $\MTCP$, in the sense that we consider only the orbit dynamics and that we assume that the attitude angles (Euler angles) can be driven to the target values instantaneously. Thus, the attitude angles are considered as control inputs in that simpler problem. Denoting the rocket axial symmetric axis as $\vec{e}$ and considering it as the control vector (which is consistent with the attitude angles $\theta$, $\psi$), we formulate the problem as follows:
$$
\dot{\vec{V}} = a \vec{e}+\vec{g}, \quad
\vec{V}(0)=\vec{V}_0,\quad \vec{V}(t_f) // \vec{w},\quad  \Vert \vec{w} \Vert =1,  \qquad \min t_f ,
$$
where $\vec{w}$ is a given vector that refers to the desired target velocity direction. 
This problem is easy to solve, and the solution is the following.

\begin{lemma} \label{OCPZ_solution}
The optimal solution of $\OCPZ$ is given by 
\begin{equation*}
 \vec{e}^{\ast}=\frac{1}{a}\left(\frac{k \vec{w}-\vec{V}_0}{t_f}-\vec{g}\right), \quad
 t_f=\frac{-a_2 + \sqrt{a_2^2-4a_1a_3}}{2 a_1},\quad
 \vec{p}_v=\frac{-p^0}{a + \langle \vec{e}^{\ast},\vec{g} \rangle } \vec{e}^{\ast}.
\end{equation*}
with 
$k=\langle \vec{V}_0,\vec{w} \rangle + \langle \vec{g},\vec{w} \rangle t_f$, 
$a_1 = a^2-\Vert  \langle \vec{g},\vec{w} \rangle \vec{w}-\vec{g} \Vert ^2$, 
$a_2 = 2 ( \langle \vec{V}_0,\vec{w} \rangle \langle \vec{g},\vec{w} \rangle - \langle \vec{V}_0,\vec{g} \rangle)$, and
$a_3 = - \Vert  \langle \vec{V}_0,\vec{w} \rangle \vec{w}-\vec{V}_0 \Vert ^2$.
\end{lemma}

\begin{proof}
The Hamiltonian is $H=p^0+\vec{p}_v (a \vec{e}+\vec{g})$, and we have $\dot{\vec{p}}_v=\vec{0}$, with $\vec{p}_v=(p_{v_x},p_{v_y},p_{v_z})^{\top}$, and $H=0$ along any extremal. It follows that $\vec{p}_v \neq \vec{0}$ (indeed otherwise we would get also $p^0=0$, and thus a contradiction). Hence there are no singular controls for this problem. The maximization condition of the PMP yields $\vec{e}^{\ast}=\vec{p}_v/\Vert \vec{p}_v\Vert $, and hence the optimal control is a constant vector. Moreover, according to the final condition $\vec{V}(t_f) // \vec{w}$, the transversality condition is $\vec{p}_v \perp \vec{w}$, hence $\langle\vec{e}^{\ast}, \vec{w} \rangle =0$, and using $\vec{V}(t_f)=\vec{V}_0+(a \vec{e}+\vec{g} ) t_f = k \vec{w}$ we get that $\vec{e}^{\ast}=\frac{1}{a}(\frac{k \vec{w}-\vec{V}_0}{t_f}-\vec{g})$.
It follows from the transversality condition that $k=\langle \vec{V}_0,\vec{w} \rangle + \langle \vec{g},\vec{w} \rangle t_f $. The expression of $t_f$ follows, using that $\Vert  \vec{e}^{\ast} \Vert ^2=1$.
Using that $H=0$, we get $\vec{p}_v=\frac{-p^0}{a + \langle \vec{e}^{\ast},\vec{g} \rangle } \vec{e}^{\ast}$.
\end{proof}

Since the vector $\vec{e}$ is expressed in the launch frame as $(\vec{e})_R = (\sin \theta \cos \psi, -\sin \psi, \cos \theta \sin \psi)^{\top} $, the Euler angles $\theta^{\ast} \in (-\pi/2,\pi/2)$ and $\psi^{\ast} \in (-\pi/2,\pi/2)$ are given by
\begin{equation} \label{sing_angles}
\theta^{\ast} = \mathrm{arctan}(e_1^{\ast}/e_3^{\ast}), \qquad \psi^{\ast}=-\mathrm{arcsin}(e_2^{\ast}),
\end{equation}
where $e_i^{\ast}$ is the $i$-th component of $\vec{e}^{\ast}$, for $i=1,2,3$.

Using the definition \eqref{ssurface} of the singular surface $S$, we check that the optimal solution of $\OCPZ$ is contained in $S$ with $\theta = \theta^\ast$, $\psi=\psi^\ast$ and $\phi=\phi^\ast$ ($\phi^\ast$ is any real number). 
Therefore, the relationship between $\OCPZ$ and $\MTCP$ is the following.

\begin{lemma}\label{lem_relationship}
The optimal solution of the problem $\OCPZ$ actually corresponds to a singular solution of $\MTCP$ with the terminal conditions given by
\begin{equation} \label{inicond_ocpztomtcp}
\begin{split}
v_x(0) = {v_{x_0}}, \quad v_y(0) = {v_{y_0}}, \quad v_z(0) = {v_{z_0}}, \\
\theta(0)=\theta^\ast, \quad \psi(0)=\psi^\ast,\quad, \phi(0)=\phi^\ast,\quad \omega_x(0) = 0,\quad \omega_y(0)=0,
\end{split}
\end{equation}
\begin{equation} \label{fincond_ocpztomtcp}
v_z(t_f)\sin \psi_f + v_y(t_f) \cos \theta_f \cos \psi_f=0,\quad v_z(t_f)\sin \theta_f-v_x(t_f)\cos \theta_f=0,
\end{equation}
\begin{equation} \label{fincond_ocpztomtcp2}
 \theta(t_f)=\theta^\ast, \quad \psi(t_f)=\psi^\ast,\quad, \phi(t_f)=\phi^\ast,\quad \omega_x(t_f) = 0,\quad \omega_y(t_f)=0.
 \end{equation}
\end{lemma}

Due to this result, a natural idea of numerical continuation strategy consists of deforming continuously (step by step) the terminal conditions given in Lemma \ref{lem_relationship}, to the terminal conditions \eqref{OCPc_ic}-\eqref{OCPc_fc} of the problem $\MTCP$. 

However, because of the chattering phenomenon, we cannot make converge the shooting method in such a strategy. More precisely, when the terminal conditions are in the neighborhood of the singular surface $S$, the optimal extremals are likely to contain a singular arc (and thus chattering arcs). In that case, the shooting method will certainly fail due to the difficulty of numerical integration of discontinuous Hamiltonian system. Hence, we introduce hereafter an additional numerical trick and we define the following regularized problem, in which we modify the cost functional with a parameter $\gamma$, so as to overcome the problem caused by chattering.

\paragraph{Regularized problem.}
Let $\gamma > 0$ be arbitrary. The regularized problem \OCPReps consists of minimizing the cost functional
\begin{equation} \label{cost_opcr}
    C_\gamma = t_f + \gamma \int_0^{t_f} (u_1^2 + u_2^2) \, dt ,
\end{equation}
for the bi-input control-affine system \eqref{sys_multi_affine}, under the control constraints $-1\leq u_i\leq 1$, $i=1,2$, with terminal conditions \eqref{OCPc_ic}-\eqref{OCPc_fc}.
Note that, here, we replace the constraint $u_1^2+u_2^2\leq 1$ (i.e., $u$ takes its values in the unit Euclidean disk) with the constraint that $u$ takes its values in the unit Euclidean square. The advantage, for this intermediate optimal control problem with the cost \eqref{cost_opcr}, is that the extremal controls are then continuous.

The Hamiltonian is
\begin{equation}\label{Hamil_n}
H_\gamma=\langle p,f(x) \rangle + u_1\langle p,g_1(x) \rangle + u_2 \langle p,g_2(x) \rangle +p^0(1+\gamma u_1^2 + \gamma u_2^2),
\end{equation}
and according to the PMP, the optimal controls are
\begin{equation} \label{OCPn_u1}
u_1(t) = \mathrm{sat}(-1,-\bar{b} p_{\omega_y}(t) / (2 \gamma p^0),1),\qquad
u_2(t) = \mathrm{sat}(-1,\bar{b} p_{\omega_x}(t) / (2 \gamma p^0),1),
\end{equation}
where the saturation operator $\mathrm{sat}$ is defined by $\mathrm{sat}(-1,f(t),1)=-1$ if $f(t)\leq -1$; $1$ if $f(t)\geq 1$; and $f(t)$ if $-1\leq f(t)\leq 1$.

\medskip

As mentioned previously, one of the motivations for considering the intermediate problem \OCPReps is that the solution of $\OCPZ$ is a \emph{singular} trajectory of the full problem $\MTCP$, and hence, passing directly from $\OCPZ$ to $\MTCP$ causes difficulties due to chattering (see Corollary \ref{cor_chattering}).  The following result shows that when we embed the solutions of $\OCPZ$ into the problem \OCPReps, they are not singular.

\begin{lemma}
\label{lem_relat_ocpzetocpr}
An extremal of $\OCPZ$ can be embedded into the problem \OCPReps by setting
$$
u(t) = (0,0), \quad \theta(t)=\theta^\ast, \quad \psi(t)=\psi^\ast, \quad \phi(t)=\phi^\ast, \quad \omega_x(t) = 0, \quad \omega_y(t)=0,
$$
$$
p_{\theta}(t)=0, \quad p_{\psi}(t)= 0, \quad p_{\phi}(t)=0, \quad p_{\omega x}(t) = 0, \quad p_{\omega y}(t)=0,
$$
where $\theta^\ast$ and $\psi^\ast$ are given by \eqref{sing_angles}, 
with terminal conditions given by \eqref{inicond_ocpztomtcp} and \eqref{fincond_ocpztomtcp}-\eqref{fincond_ocpztomtcp2}. Moreover, it is not a singular extremal for the problem \OCPReps.
The extremal equations for \OCPReps are the same than for $\MTCP$, as well as the transversality conditions.
\end{lemma}

\begin{proof}
It is easy to verify that the embedded extremal is an extremal of the problem \OCPReps and that the transversality conditions are the same. 
The control is computed from \eqref{OCPn_u1} which maximizes the Hamiltonian $H_{\gamma}$, and we have $H_{\gamma} = 0$ with $p^0=-1$. It follows from the PMP that the extremal equations are the same than for $\MTCP$.
Then, for the problem \OCPReps, we have $\frac{\partial^2 H_{\gamma} }{\partial u_i^2}= \gamma p^0 $. Note that, in this case, the control $u_i$, $i=1,2$ is singular if $\frac{\partial^2 H_{\gamma} }{\partial u_i^2}= 0$. Hence there is no normal singular extremal for the problem \OCPReps. From Lemma \ref{OCPZ_solution}, it is easy to see that $p^0 \neq 0$ and thus the extremals of $\OCPZ$ are not singular extremals of \OCPReps.
\end{proof}

\subsubsection{Strategy for solving $\MTCP$}
\label{indirecte_strategy}
\paragraph{Continuation procedure.}
The ultimate objective is to compute the optimal solution of the problem $\MTCP$, starting from the explicit, simple to compute, solution of $\OCPZ$. We proceed as follows:
\begin{itemize}
\item First, according to Lemma \ref{lem_relat_ocpzetocpr}, we embed the solution of $\OCPZ$ into \OCPReps.
For convenience, we still denote by $\OCPZ$ the problem $\OCPZ$ seen in high dimension. 
\item Then, we pass from $\OCPZ$ to $\MTCP$ by means of a numerical continuation procedure, involving three continuation parameters: the first two parameters $\lambda_1$ and $\lambda_2$ are used to pass continuously from the optimal solution of $\OCPZ$ to the optimal solution of the regularized problem \OCPReps, for some fixed $\gamma>0$, and the third parameter $\lambda_3$ is then used to pass to the optimal solution of $\MTCP$ (see Figure \ref{homotopie}).
\end{itemize}

\begin{figure}[h]
\centering 
	\includegraphics[scale=0.8]{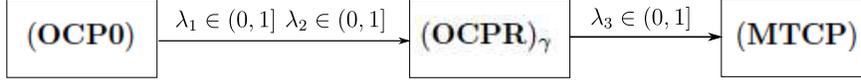}
	\caption{Continuation procedure.}
	\label{homotopie}
\end{figure}

The parameter $\lambda_1$ is used to act, by continuation, on the initial conditions, according to
$$
\theta(0) = \theta^{\ast} (1-\lambda_1) + \theta_0 \lambda_1, \quad \psi(0) = \psi^{\ast} (1-\lambda_1) + \psi_0 \lambda_1, \quad \phi(0) = \phi^{\ast} (1-\lambda_1) + \phi_0 \lambda_1, 
$$
$$
\omega_x(0) = \omega_x^{\ast} (1-\lambda_1) + {\omega_{x_0}} \lambda_1, \quad \omega_y(0) = \omega_y^{\ast} (1-\lambda_1) + {\omega_{y_0}} \lambda_1,
$$
where $\omega_x^{\ast} = \omega_y^{\ast} =0$, $\phi^{\ast} = 0$, and $\theta^{\ast}$, $\psi^{\ast}$ are calculated through equation \eqref{sing_angles}.

Using the transversality condition \eqref{trc} and the extremal equations $\dot{p}_{v_x}=0$, $\dot{p}_{v_y}=0$ and $\dot{p}_{v_z}=0$, 
the unknown $p_{v_y}$ can be expressed in terms of $p_{v_x}$ and $p_{v_z}$ as
$$
 p_{v_y} = (p_{v_x} \sin \theta_f \cos \psi_f+ p_{v_z} \cos \theta_f \cos \psi_f)/\sin \psi_f ,
$$
and hence the unknowns of the shooting problem are reduced to $p_{v_x}$, $p_{v_z}$, $p_{\theta}(0)$, $p_{\psi}(0)$, $p_{\phi}(0)$, $p_{\omega_x}(0)$, $p_{\omega_y}(0)$ and $t_f$. 
The shooting function $ \d{S_{\lambda_1}} $ for the $\lambda_1$-continuation is defined by
\begin{equation*}
\begin{split}
S_{\lambda_1} = \big(
        &  p_{\omega_x }(t_f) , \,\,
            p_{\omega_y} (t_f),\,\,
            p_{\theta} (t_f),\,\,
            p_{\psi} (t_f),\,\,
            p_{\phi} (t_f),\,\,
           H_\gamma(t_f), \\
        &  v_z(t_f) \sin \psi_f + v_y(t_f) \cos \theta_f \cos \psi_f,\,\,
         v_z(t_f) \sin \theta_f - v_x(t_f) \cos \theta_f \big),
\end{split}
\end{equation*}
where $H_\gamma(t_f)$ with $p^0=-1$ is calculated from \eqref{Hamil_n} and $u_1$ and $u_2$ are given by \eqref{OCPn_u1}. 
In fact, from Lemma \ref{lem_singulararcs}, we know that a singular extremal of problem $\MTCP$ must be normal, and since we are starting to solve the problem from a singular extremal, here we assume that $p^0 = -1$. 

Note that we can use $S_{\lambda_1}$ as shooting function thanks for \OCPReps. For problem $\MTCP$, if $S_{\lambda_1}=0$, then together with $\omega_x(t_f)=0$ and $\omega_y(t_f)=0$, the final point $(x(t_f),p(t_f))$ of the extremal is then lying on the singular surface $S$ defined by \eqref{ssurface} and this will cause the fail of the shooting. However, for problem \OCPReps, even when $x(t_f) \in S$, the shooting problem can still be solved.

Initializing with the solution of $\OCPZ$, we can solve this shooting problem with $\lambda_1=0$, and we get a solution of \OCPReps with the terminal conditions \eqref{inicond_ocpztomtcp}-\eqref{fincond_ocpztomtcp}  (the other states at $t_f$ being free). Then, by continuation, we make $\lambda_1$ vary from $0$ to $1$, and in this way we get the solution of \OCPReps for $\lambda_1=1$. With this solution, we can integrate extremal equations \eqref{sys_full} and \eqref{sys_adjoint} to get the values of the state variable at $t_f$. Then denote $\theta_e := \theta(t_f)$, $\psi_e := \psi(t_f)$, $\phi_e := \phi(t_f)$, $\omega_{xe} := \omega_x(t_f)$ and $\omega_{ye} := \omega_y(t_f)$. 

\medskip

In a second step, we use the continuation parameter $\lambda_2$ to act on the final conditions, in order to make them pass from the values $\theta_e$, $\psi_e$, $\phi_e$, $\omega_{xe}$ and $\omega_{ye}$, to the desired target values $\theta_f$, $\psi_f$, $\phi_f$, $\omega_{xf}$ and $\omega_{yf}$. The shooting function is
\begin{equation*}
\begin{split}
 S_{\lambda_2} = \big( &
        \omega_x(t_f) - (1-\lambda_2) \omega_{xe} -\lambda_2 \omega_{x_f}, \,
        \omega_y(t_f) - (1-\lambda_2) \omega_{ye} -\lambda_2 \omega_{y_f},\\
       & \theta(t_f)-(1-\lambda_2) \theta_e - \lambda_2 \theta_f ,\,
        \psi(t_f)-(1-\lambda_2) \psi_e - \lambda_2 \psi_f,\,
        \phi(t_f)-(1-\lambda_2) \phi_e - \lambda_2 \phi_f ,\\
       & v_z(t_f) \sin \psi_f + v_y(t_f) \cos \theta_f \cos \psi_f,\,
        v_z(t_f) \sin \theta_f - v_x(t_f) \cos \theta_f ,\,
        H_\gamma(t_f) \big) .
\end{split}
\end{equation*}
Solving this problem by making vary $\lambda_2$ from $0$ to $1$, we obtain the solution of \OCPReps with the terminal conditions \eqref{OCPc_ic}-\eqref{OCPc_fc}.

\medskip

Finally, in order to compute the solution of $\MTCP$, we use the continuation parameter $\lambda_3$ to pass from \OCPReps to $\MTCP$. We add the parameter $\lambda_3$ to the Hamiltonian $H_{\gamma}$ and to the cost functional \eqref{cost_opcr} as follows:
$$
C_\gamma = t_f + \gamma \int_0^{t_f} (u_1^2 + u_2^2)(1-\lambda_3) \, dt ,
$$
$$
H(t_f,\lambda_3)=\langle p,f \rangle+\langle p,g_1 \rangle u_1 +\langle p,g_2 \rangle u_2 +p^0 + p^0\gamma( u_1^2 +u_2^2)(1-\lambda_3).
$$
Then, according to the PMP, the extremal controls are given by $u_i=\mathrm{sat}(-1,u_{ie},1)$, $i=1,2$, where
$$
u_{1e} =   \frac{\bar{b} p_{\omega_y}}{-2 p^0 \gamma (1-\lambda_3) + \bar{b}  \lambda_3 \sqrt{p_{\omega_x}^2+p_{\omega_y}^2} } ,\quad
u_{2e} = \frac{ -\bar{b} p_{\omega_x}}{-2 p^0 \gamma (1-\lambda_3) + \bar{b}  \lambda_3 \sqrt{p_{\omega_x}^2+p_{\omega_y}^2} } .
$$
The shooting function $S_{\lambda_3}$ is defined as $S_{\lambda_2}$, replacing $H_\gamma(t_f)$ with $H_\gamma(t_f,\lambda_3)$. The solution of $\MTCP$ is then obtained by making vary $\lambda_3$ continuously from $0$ to $1$. 

\medskip

\begin{remark}\label{rem_subopti}
Note that the above continuation procedure fails in case of chattering (see Corollary \ref{cor_chattering}), and thus cannot be successful for any possible choice of terminal conditions. In particular, if chattering occurs then the $\lambda_3$-continuation is expected to fail for some value $\lambda_3 = \lambda_3^\ast<1$. But in that case, with this value of $\lambda_3$, we have generated a sub-optimal solution of the problem $\MTCP$, which appears to be acceptable and very interesting for practice. Moreover, the overall procedure is very fast and accurate. Note that the resulting sub-optimal control is continuous.
\end{remark}

\subsection{Direct method}\label{C_dc}
We now propose a direct approach for solving the problem $\MTCP$, where the control is approximated by a piecewise constant control over a given time subdivision.
The solutions derived from such a method are therefore sub-optimal, in particular when the control is chattering (and in such a case the number of switchings is limited by the time step). 
Note that this approach is much more computationally demanding than the indirect one. 

Since the initialization of a direct method may also raise some difficulties, we propose the following strategy. The idea is to start from the solution of the problem $\MTCP$ with less terminal requirements, which is easy to obtain with a direct method, and then we introduce step by step the final conditions \eqref{OCPc_fc} of the problem $\MTCP$. We implement this direct approach with the software \texttt{BOCOP} and its batch optimization option (see \cite{BonnansMartinon}). 

\begin{itemize}
\item Step 1: we solve the problem $\MTCP$ with initial conditions \eqref{OCPc_ic} and final conditions
$$
\omega_{y}(t_f)=0,\quad \theta (t_f) = \theta_f ,\quad v_{z} (t_f) \sin \theta_f - v_{x} (t_f) \cos \theta_f =0.
$$
These final conditions are the ones of the planar version of $\MTCP$ in which the motion of the spacecraft is 2D (see \cite{ZTC} for details). 
Numerical simulations show that, with such terminal conditions, the problem $\MTCP$ is easy and fast to solve by means of a direct method (a constant initial guess for the discretized variables suffices to ensure convergence). 

\item Then, in Steps 2, 3, 4 and 5, we add successively (and step by step) the final conditions
$v_{z} (t_f) \sin \psi_f + v_{y} (t_f) \cos \theta_f \cos \psi_f =0$, $\psi (t_f)= \psi_f$, $\phi (t_f)=\phi_f$, and $ \omega_{x}(t_f) = \omega_{xf}$, and for each new step we use the solution of the previous one as an initial guess.
\end{itemize}
At the end of this process, we have obtained the solution of the full problem $\MTCP$. Note again that this direct approach is much slower than the indirect one, and that the resulting control has many numerical oscillations (see numerical results in Section \ref{C_ca}).

\section{Numerical results} \label{Chp_numerical}
The structure of the rocket is presented in Figure \ref{Thrust} (b). We assume that the thrust $I$ is flexible, i.e., it can turn $\pm 6^{\circ}$ in all directions, and its thrust is around $T_{att}=1400$ kN. The other thrusts are fixed with a total thrust $T_{tot}=1 \times 10^5$ kN. 
The rocket mass is $800$ t, the length of the rocket is $l_r=50$ m and its radius is $r_r=2.5$ m. Considering the rocket as a cylinder, we have $I_x=I_y=m(3 r_r^2+l_r^2 )/12$ and $I_z=m r_r^2/2$.
The parameters $a$ and $b$ in \eqref{orbit_dynamics} and \eqref{attitude_dynamics} are therefore $a=T_{tot}/m \approx 12$ and $\bar{b}=\frac{ T_{att} l_r}{2 I_x} \mu_{max}\approx 0.02$.

During the atmospheric ascent phase, the velocity of the rocket remains between several hundreds m/s and around 1000 m/s. Let $v = \sqrt{v_x^2+v_y^2+v_z^2}$ be the modulus of the velocity, and let $\psi_v$ and $\phi_v$ be the flight path angles that we use to calculate the components of the velocity in $S_R$ frame, i.e., $v_x = v \sin \theta_v \cos \psi_v$, $v_y = -v \sin \psi_v$ and $v_z = v \cos \theta_v \cos \psi_v$.
In this section, the initial values of the angles $\theta_v$ and $\psi_v$ are chosen equal to the initial values of the angles $\theta$ and $\psi$. This means that, before the maneuver, the rocket is on a trajectory with angle of attack equal to zero.

In the numerical simulations, we set ${v_{x_0}} = v_0 \sin \theta_0 \cos \psi_0$, ${v_{y_0}} = - v_0 \sin \psi_0$, ${v_{z_0}} = v_0 \cos \theta_0 \cos \psi_0$ and take the other values needed in the initial condition \eqref{OCPc_ic} and the final condition \eqref{OCPc_fc} in the following table
\begin{table} [H] 
\centering
\begin{threeparttable}[b]
    \begin{tabular}{ |l | l| l| l| l| l| l| l| l| l|}
    \hline
    \multicolumn{9}{c}{(TC1): ${\omega_{x_0}} = {\omega_{y_0}}=0$, $\theta_0=75^{\circ}$, $\psi_0 = 0.5^{\circ}$, $\phi_0 = 0^{\circ} $ }  \\ 
    \multicolumn{9}{c}{  $\omega_{x_f} = \omega_{y_f}=0$, $\theta_f=85^{\circ}$, $\psi_f = 5^{\circ}$, $\phi_f=0^{\circ}$ }  \\ 
    \hline
    \multicolumn{9}{c}{(TC2): ${\omega_{x_0}} = {\omega_{y_0}}=0$, $\theta_0=70^{\circ}$, $\psi_0 = 0.5^{\circ}$, $\phi_0 = 0^{\circ} $ }  \\ 
    \multicolumn{9}{c}{  $\omega_{x_f} = \omega_{y_f}=0$, $\theta_f=85^{\circ}$, $\psi_f = 5^{\circ}$, $\phi_f=0^{\circ}$ }  \\ 
     \hline   
    \multicolumn{9}{c}{(TC3): ${\omega_{x_0}} = {\omega_{y_0}}=0$, $\theta_0=85^{\circ}$, $\psi_0 = 0.5^{\circ}$, $\phi_0 = 0^{\circ} $ }  \\ 
    \multicolumn{9}{c}{  $\omega_{x_f} = \omega_{y_f}=0$, $\theta_f=75^{\circ}$, $\psi_f = 5^{\circ}$, $\phi_f=0^{\circ}$ }  \\ 
     \hline   
    \end{tabular}
\end{threeparttable}
\caption{ \label{tercond_table} Terminal conditions}
\end{table}
Note that $v_0$ is the module of velocity at time $0$. In the next two subsections, we will choose different values of $v_0$, and so here we do not assign to it a specific value.
Moreover, we set $\gamma=50$ as the weight of the $L_2$-norm control term in the cost functional of problem \OCPReps.

\subsection{Numerical results without chattering}
\label{C_nric}
The indirect method combined with numerical continuation described in Section \ref{C_ic} is implemented using a predictor-corrector continuation method, where the prediction is made thanks to a Lagrange polynomial. The \texttt{Fortran} routines \texttt{hybrd.f} (see \cite{More}) and \texttt{dop853.f} (see \cite{Hairer}) are used, respectively, for solving the shooting problem (Newton method) and for integrating the ordinary differential equations (with prediction).

The Euler angle $\theta$ is usually called the \emph{pitch angle}, and a \emph{pitching up maneuver} designates a maneuver with terminal condition $\theta_f > \theta_0$, while a \emph{pitching down maneuver} designates a maneuver with terminal condition $\theta_f<\theta_0$. 

\paragraph{Pitching up maneuvers.}
We set $v_0=1000$ m/s and we use the numerical values denoted by (TC1) in Table \ref{tercond_table}. The components of the state variable are reported on Figure \ref{state_1000}. The optimal control, the adjoint variables $p_{\omega_x}(t)$ and $p_{\omega_y}(t)$ and the modulus of the switching function $\Phi(t)=\bar{b}(p_{\omega_y},-p_{\omega_x})$ are reported on Figure \ref{adjoint_1000}.
\begin{figure}[h]
	\includegraphics[scale=0.36]{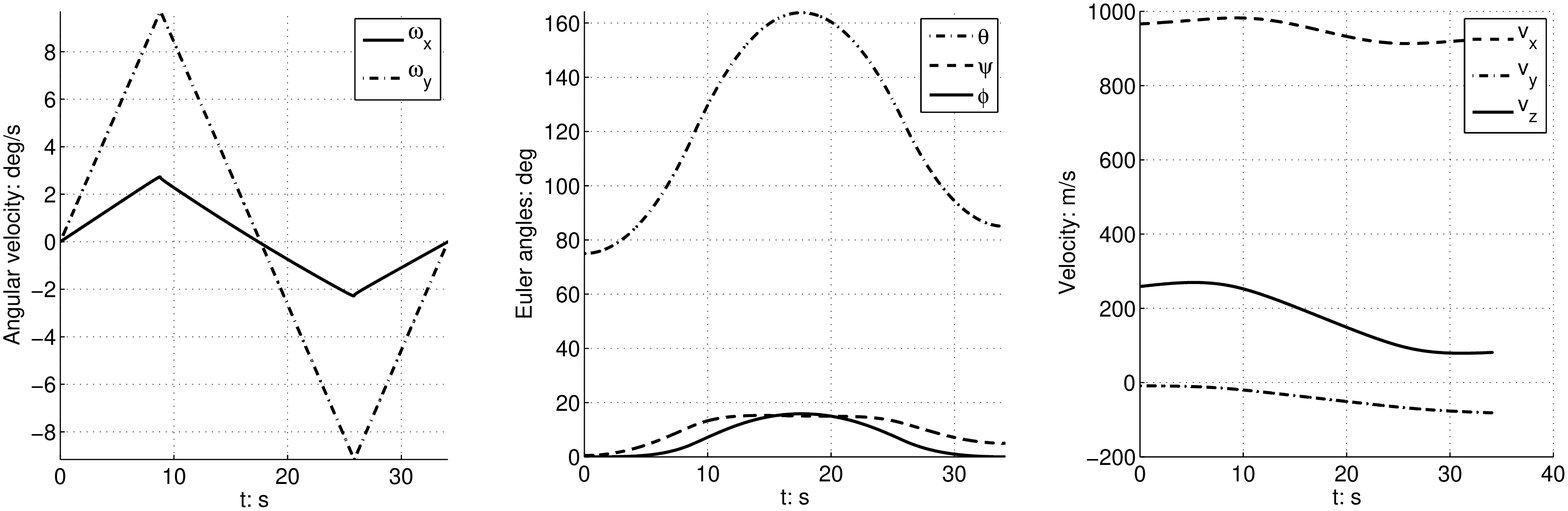}
	\caption{State variable and the optimal control with (TC1) and $v_0=1000$.}
	\label{state_1000}
\end{figure}
\begin{figure}[h]
\centering
	\includegraphics[scale=0.36]{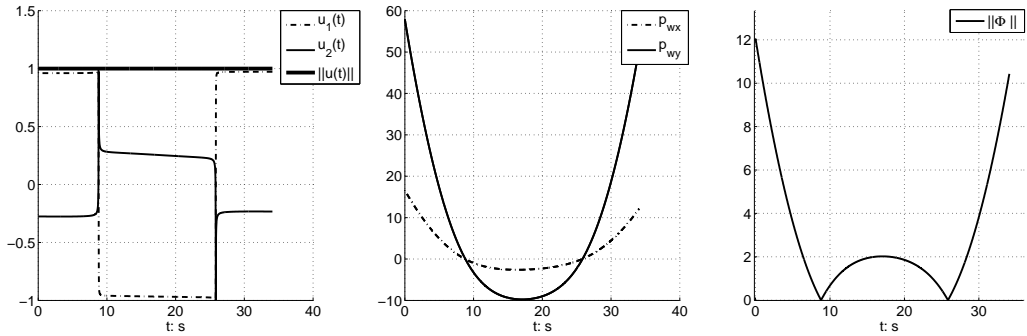}
	\caption{Adjoint variable and the switching function with (TC1) and $v_0=1000$.}
	\label{adjoint_1000}
\end{figure}
We observe that the optimal control switches twice, at times $8.8$ s and $25.8$ s. These two switching points are of order $1$ (i.e., $\Phi(t)=0$ and $\dot{\Phi}(t)\ne0$). Accordingly with Lemma \ref{lem_pointorder1}, the control turns with an angle $\pi$ at those points.

\medskip

Let us give another numerical example, taking the same terminal conditions as previously except for $v_0$, and we take $v_0=1500$ m/s. The time history of the state, of the optimal control and of the switching function are reported on Figures \ref{state_1500} and \ref{adjoint_1500}. One can see on Figure \ref{adjoint_1500} that the optimal control turns two more times with an angle $\pi$ due to two switching points of order one. 
\begin{figure}[h]
        \centering
        \includegraphics[scale=0.36]{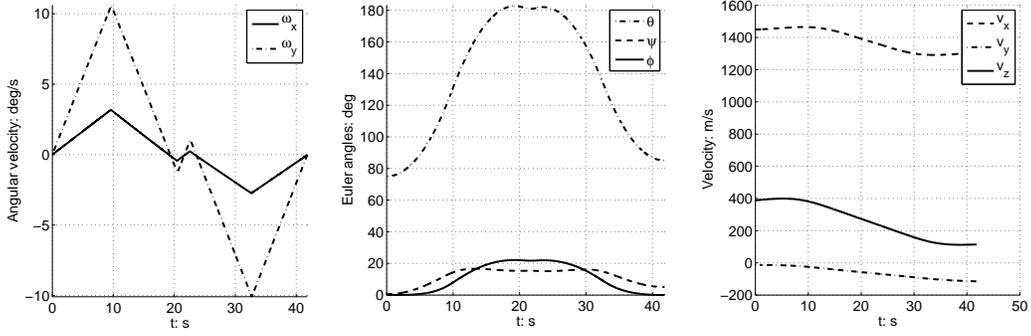}
        \caption{Time histories of state with (TC1) and $v_0=1500\, m/s$.}
        \label{state_1500}
\end{figure}
\begin{figure}[h]
        \centering
        \includegraphics[scale=0.36]{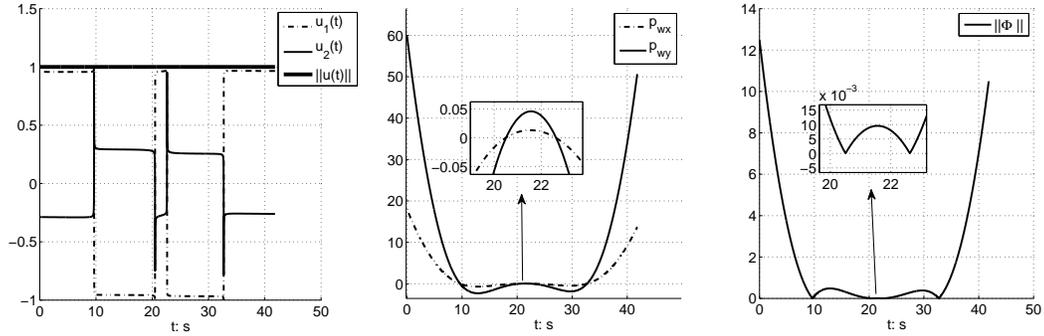}
        \caption{Time histories of control and switching function with (TC1) and $v_0=1500\, m/s$.}
        \label{adjoint_1500}
\end{figure}

\paragraph{Pitching down maneuvers.}
We set $v_0=1500$ m/s and we use the numerical values denoted by TC3 in Table \ref{tercond_table}. The optimal solution is drawn on Figures \ref{state_1500_inverse} and \ref{control_1500_inverse}.
\begin{figure}[h]
        \centering
        \includegraphics[scale=0.36] {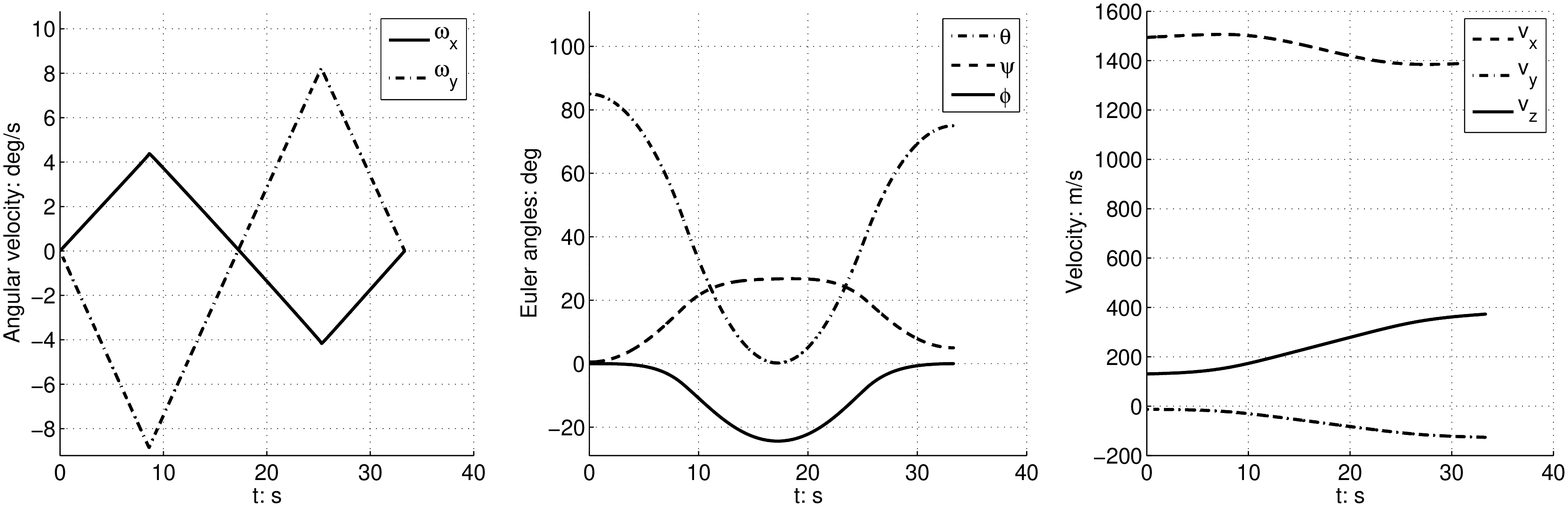}
        \caption{Time histories of state with TC3 and $v_0=1500$.}
        \label{state_1500_inverse}
\end{figure}
\begin{figure}[h]
        \centering
        \includegraphics[scale=0.36] {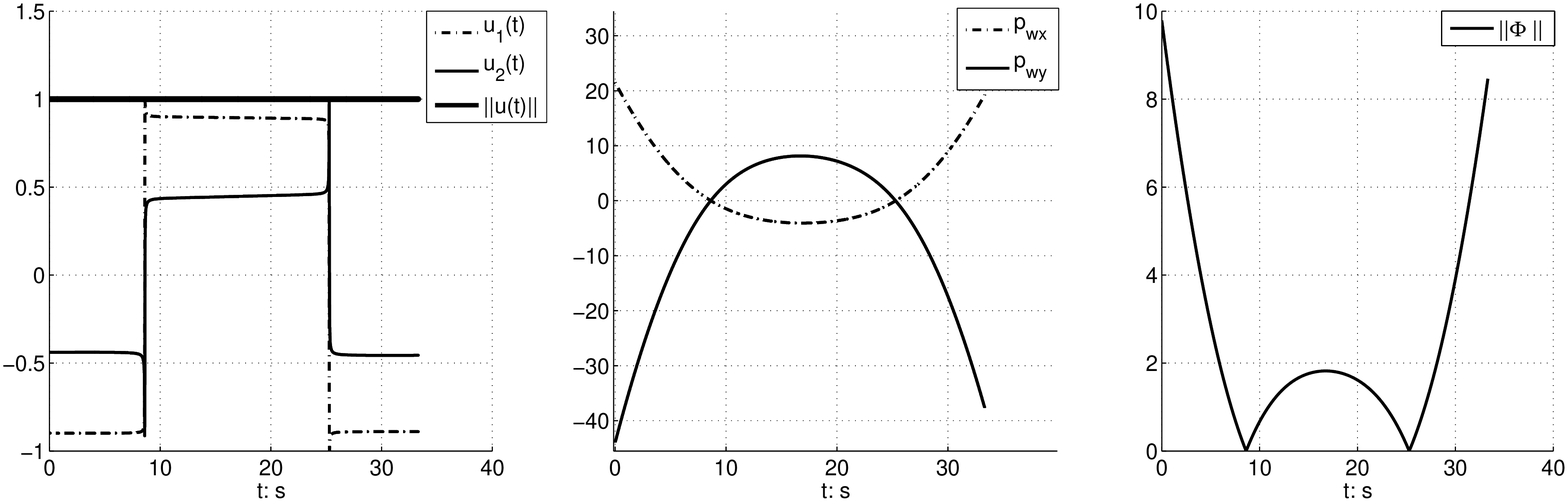}
        \caption{Time histories of control and switching function with TC3 and $v_0=1500$.}
        \label{control_1500_inverse}
\end{figure}
The shorter maneuver time $t_f$ indicates that it is easier to turn clockwise the axis of the velocity vector than to turn it anti-clockwise. This corresponds to the intuition. The reason is that the total force induced by the gravity force tends to reduce $v_x$, i.e., it helps the velocity to turn clockwise, and so together with the rocket thrust force, the maneuver time is less than that of the anti-clockwise case.

Note that the derived time history of the adjoint variable (for both pitching up and pitching down maneuvers) do not have the same order of magnitude, i.e., $p_{\omega_x}$ and $p_{\omega_y}$ are ten times larger than $p_{\theta}$ and $p_{\psi}$, and are thousand times larger than $p_{v_x}$, $p_{v_y}$ and $p_{v_z}$. This indicates again that the shooting method is difficult to initialize successfully. 

We note that the indirect strategy proposed in Section \ref{indirecte_strategy} is efficient also because the smallest adjoint variables $p_{v_x}$, $p_{v_y}$ and $p_{v_z}$ are already quite accurately estimated thanks to the problem of order zero.

\subsection{Numerical results with chattering arcs}\label{C_ca}
\paragraph{Sub-optimal solution by the indirect approach.}
On Figures \ref{state_2000} and \ref{adjoint_2000} is given a sub-optimal solution of $\MTCP$ with the terminal conditions (TC2) of Table \ref{tercond_table} and $v_0=2000$ m/s. Due to chattering, the continuation parameter $\lambda_3$ stops at value $\lambda_3^\ast=0.98$ (see Remark \ref{rem_subopti}). 
\begin{figure}[h]
        \centering
        \includegraphics[scale=0.36]{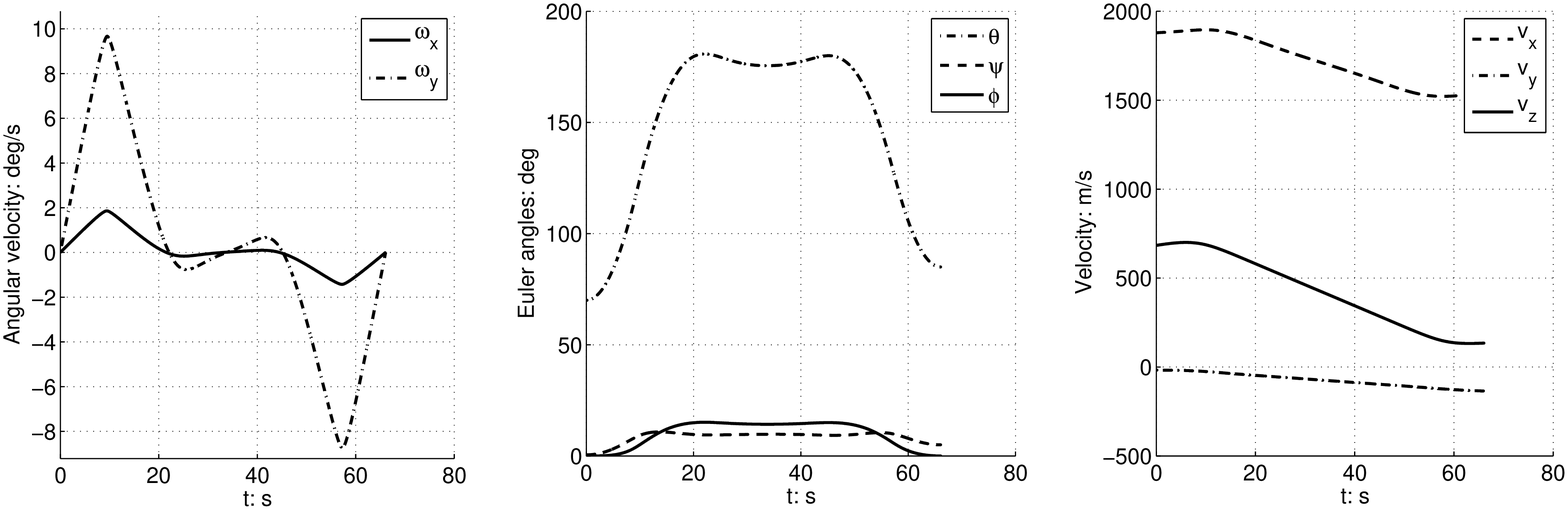}
        \caption{Time histories of state with (TC2) and $v_0=2000$.}
        \label{state_2000}
\end{figure}
\begin{figure}[h]
        \centering
        \includegraphics[scale=0.36]{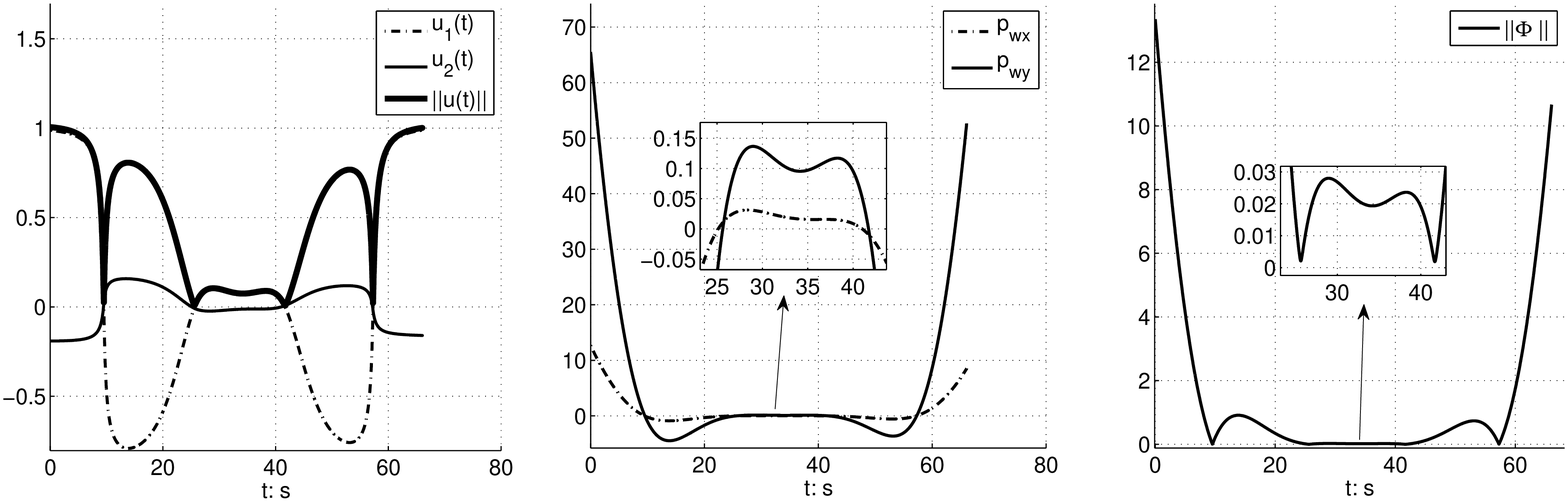}
        \caption{Time histories of control and switching function with (TC2) and $v_0=2000$.}
        \label{adjoint_2000}
\end{figure}
Observing from Figure \ref{adjoint_2000}, the switching function pass four times the switching surface $\Gamma$ are small between time $26.5$ and $40.3$. The control, instead of bang-bang or singular, is continuous. The cost of this trajectory is $69.3$ and the final time $t_f = 66.0\,s$.

\paragraph{Sub-optimal solution by the direct approach.}
With the same terminal conditions as above, we now use the direct method described in Section \ref{C_dc}.
Numerical simulations show that the initialization step for the direct method procedure is quite robust (a constant initial guess is enough).
The results are reported on Figures \ref{state_2000_bocop}, \ref{control_2000_bocop} and \ref{adjoint_2000_bocop}.

\begin{figure}[h]
\centering
\includegraphics[scale=0.36]{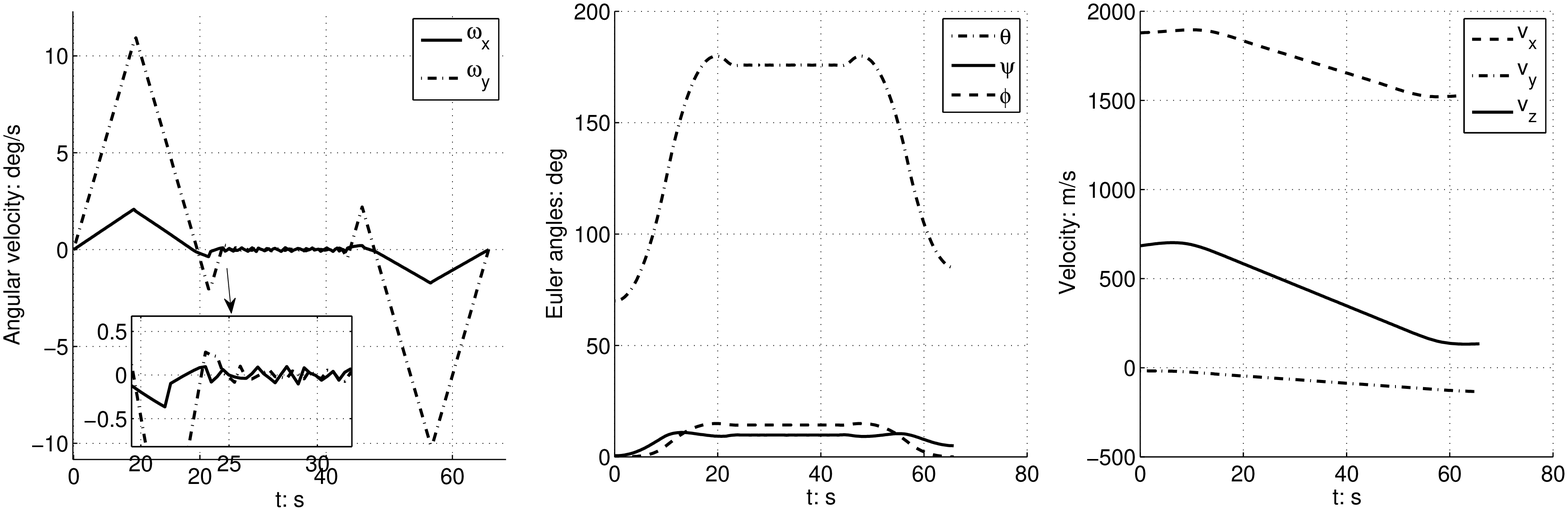}
\caption{State variable $x(t)$ ($v_0=2000\,m/s$).}
\label{state_2000_bocop}
\end{figure}
\begin{figure}[h]
\centering
\includegraphics[scale=0.36]{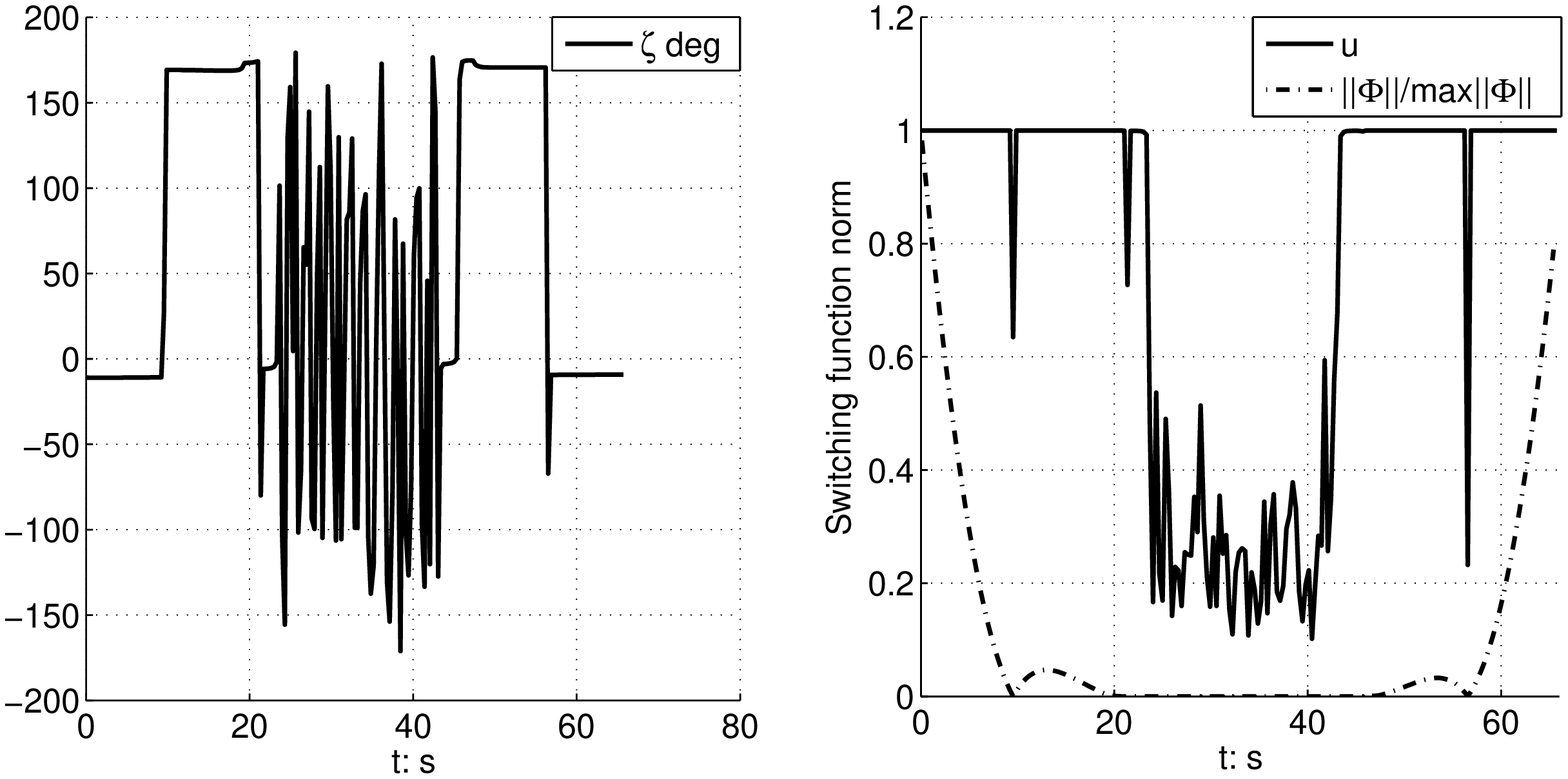}
\caption{Optimal control and $\Vert  \Phi(t) \Vert $ ($v_0=2000\,m/s$).}
\label{control_2000_bocop}
\end{figure}

We observe that, when $t \in [23,43]$, the control oscillates much with a modulus less than $1$: this indicates that there is a singular arc in the ``true" optimal trajectory, and therefore chattering according to Corollary \ref{cor_chattering}.

\begin{figure}[h]
\centering
\includegraphics[scale=0.36]{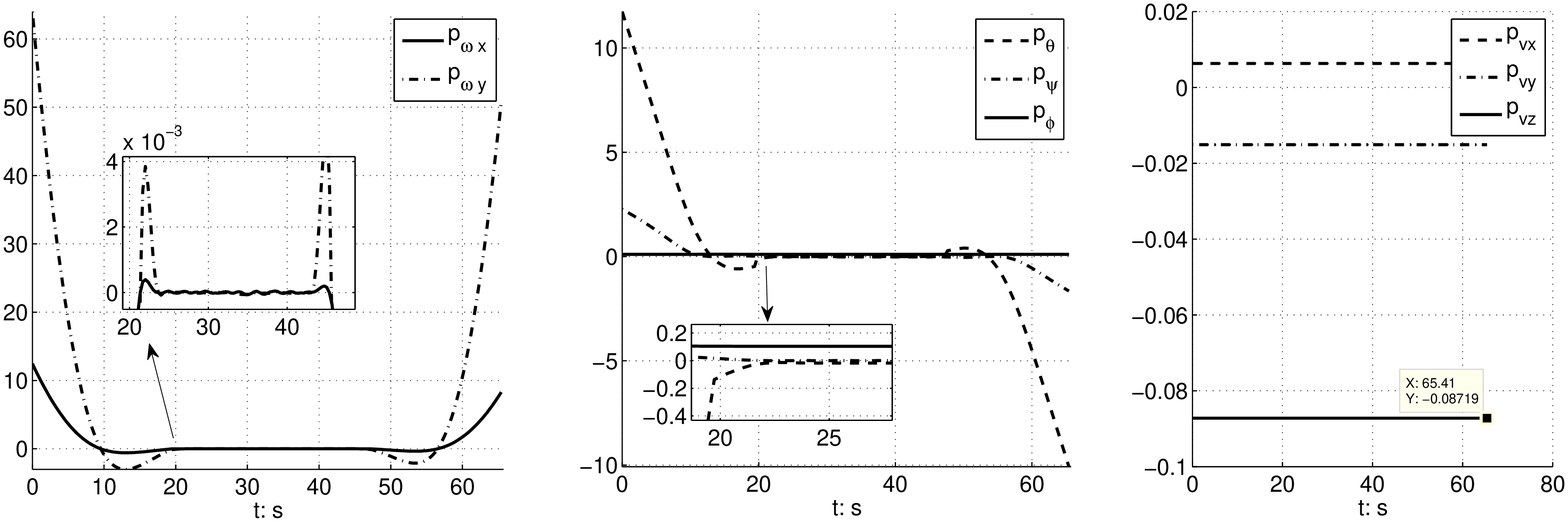}
\caption{Adjoint state $p(t)$ ($v_0=2000\,m/s$).}
\label{adjoint_2000_bocop}
\end{figure}

Note that, along the singular arc, the variables $\omega_x$, $\omega_y$, $p_{\omega_x}$, $p_{\omega_y}$, $p_\theta$, $p_\psi$ and $p_\phi$ are almost equal to $0$,
and we check that this arc indeed lives on the singular surface $S$ defined by \eqref{ssurface}. 
Therefore, it turns out that there is a singular arc in the optimal trajectory, causing chattering at the junction with regular arcs.

The maneuver time is $t_f=65.4$ s. Compared with that of the sub-optimal solution derived from the indirect strategy, only $0.6$ s are gained with the direct method. The direct approach is hundreds of times slower than the indirect approach and the obtained control presents many oscillations, which is not much appropriate for a practical use.

\medskip

On Figure \ref{state_1000}, \ref{state_1500}, \ref{state_1500_inverse}, \ref{state_2000} and \ref{state_2000_bocop}, we note that the attitude angles first tend to reach the values $\theta^\ast$ deg, $\psi^\ast$ (i.e., $\theta^\ast=176.9$ deg and $\psi^\ast=18.5$ deg for Figures \ref{state_1000} and \ref{state_1500}; $\theta^\ast=-17.4$ deg and $\psi^\ast=25.2$ deg for Figure \ref{state_1500_inverse}; $\theta^\ast=176.1$ deg and $\psi^\ast=11.2$ deg for Figures \ref{state_2000} and \ref{state_2000_bocop}), and then turn back to reach their final values. Actually, doing more numerical simulations with different terminal conditions (note reported here), we observe that the extremals have a trend to first go towards the singular surface and then to get back to the target submanifold. We suspect that this is due to a turnpike phenomenon as described in \cite{TrelatZuazua_JDE2015}, at least when the required transfer time is quite large.

\section{Conclusion}\label{Chp_conclusion}
We have studied the time optimal control of the rocket attitude motion combined with the orbit dynamics. The problem $\MTCP$ is of interest because of the coupling of guidance and navigation systems. However, this problem is difficult to solve because of the occurence of the chattering phenomenon for certain terminal conditions. 

Using geometric control, we have established a chattering result for bi-input control-affine systems. We have also classified the switching points for the extremals of the problem $\MTCP$, according to the order of vanishing of the switching function, showing the behavior of the control at the singularities.

In order to compute numerically the solutions of problem $\MTCP$, we have implemented two approaches. The indirect approach, combining shooting and numerical continuation, is time-efficient when the solution does not contain any singular arcs.
For certain terminal conditions, 
the optimal solution of $\MTCP$ involves a singular arc that is of order two, and the connection with regular arcs can only be done by means of chattering. The occurrence of chattering causes the failure of the indirect approach. For such cases, we have proposed two possible numerical alternatives. Since our indirect approach involves three continuations, one of them being concerned with a continuation on the cost function (and thus on the Hamiltonian and the control), we have proposed, as a first alternative, to stop this last continuation before its failure: in such a way, we obtain a sub-optimal solution, which seems to be very acceptable for a practical use. The second alternative is based on a direct approach, and then we obtain as well a sub-optimal solution having a finite number of switchings, this finite number being limited by the chosen step of the subdivision in the discretization scheme.
In any case, the direct strategy is much more time consuming than the indirect approach.
Note that, in both cases, it is not required to know a priori the structure of the optimal solution (in particular, the number of switchings).

As an open issue, one may consider to add atmospheric forces in the model. 
Since the magnitude of the aero-forces is low (at least, it should be much smaller than the rocket thrust), we expect this extension to be doable, for instance by means of an additional continuation.

\medskip

\paragraph{Acknowledgment.}
The second author acknowledges the support by FA9550-14-1-0214 of the EOARD-AFOSR.



\begin{thebibliography}{33}
\addcontentsline{toc}{section}{References}
\bibitem{AGRACHEV} 
A.A. Agrachev, Y.L. Sachkov, 
Control theory from the geometric viewpoint,
\textit{Springer}, 2004.

\bibitem{Berend} 
N. B\'{e}rend, F. Bonnans, M. Haddou, J. Laurent-Varin, C. Talbot,
An interior-point approach to trajectory optimization,
\textit{Journal of Guidance, Control and Dynamics}, 2007, vol. 30, no. 5, p. 1228-1238.

\bibitem{Betts} J.T. Betts,
Practical methods for optimal control and estimation using nonlinear programming,
Second edition, Advances in Design and Control, 19, 
\textit{Society for Industrial and Applied Mathematics (SIAM)}, Philadelphia, PA, 2010.

\bibitem{Bilimoria} K. D. Bilimoria, B. Wie, 
Time-optimal three-axis reorientation of a rigid spacecraft, 
\textit{Journal of Guidance, Control, and Dynamics}, 1993, vol. 16, no. 3, p. 446-452.

\bibitem{BonnansMartinon} 
F. Bonnans, P. Martinon, V. Gr\'elard,
Bocop-A collection of examples, 2012.

\bibitem{BonnardCaillau}
B. Bonnard, J.B. Caillau, E. Tr\'{e}lat,
Geometric optimal control of elliptic Keplerian orbits,
\textit{Discrete and Continuous Dynamical Systems series S}, 2005, p. 929--956.
 
\bibitem{BonnardTrelat} 
B. Bonnard, E. Tr\'{e}lat,
Une approche g\'{e}om\'{e}trique du contr\^{o}le optimal de l'arc atmosph\'{e}rique de la navette spatiale,
\textit{ESAIM: Control, Optimisation and Calculus of Variations}, 2002, vol. 7, p. 179-222.

\bibitem{BonnardChyba} B. Bonnard, M. Chyba, 
Singular trajectories and their role in control theory,
\textit{Springer Science and Business Media}, 2003, Vol. 40, 20.

\bibitem{Bryson} A.E. Bryson, 
Applied optimal control: optimization, estimation and control,
\textit{CRC Press}, 1975.

\bibitem{CHT}
M. Cerf, T. Haberkorn, E. Tr\'elat,
Continuation from a flat to a round Earth model in the coplanar orbit transfer problem,
Optimal Control Appl. Methods, 33 (2012), no. 6, 654--675.

\bibitem{Cesari} L. Cesari,
Optimization - theory and applications. Problems with ordinary differential equations,
Applications of Mathematics, 17, New York: \textit{Springer Verlag}, 1983.

\bibitem{Chitour} Y. Chitour, F. Jean, E. Tr\'elat, 
\emph{ Singular trajectories of control-affine systems},
SIAM Journal on Control and Optimization, 47(2), 1078-1095, 2008.

\bibitem{FULLER1} A. T. Fuller, 
An Optimum Non-Linear Control System, 1961,
\textit{In the Proceedings of IFAC Congress}, Moscow, USSR.

\bibitem{Gabasov} R. Gabasov, F.M. Kirillova,  
High order necessary conditions for optimality, 
\textit{SIAM Journal on Control}, 1972, vol. 10, no. 1, 127-168.

\bibitem{Gergaud} J. Gergaud, T. Haberkorn, P. Martinon,
Low thrust minimum fuel orbital transfer: an homotopic approach,
\textit{Journal of Guidance, Control and Dynamics}, 2004, vol. 27, no. 6, p. 1046-1060.

\bibitem{Goh} B.S. Goh, 
Necessary conditions for singular extremals involving multiple control variables
\textit{SIAM Journal on Control}, 1966, 4(4), 716-731. 

\bibitem{Guzzetti} D. Guzzetti, R. Armellin, M. Lavagna, 
Coupling Attitude and Orbital Motion of Extended Bodies In The Restricted Circular 3-Body Problem: A Novel Study On Effects And Possible Exploitations, 
\textit{In the Proceedings of 63rd International Astronautical Congress}, 2012.

\bibitem{Hairer} E. Hairer, S.P. Norsett, G. Wanner, 
Solving Ordinary Differential Equations I. Nonstiff Problems,
\textit{Springer series in computational mathematics, Springer-Verlag}, 1993.


\bibitem{Kelley} H.J. Kelley, R.E. Kopp, H.G. Moyer,
Singular extremals, 
1967

\bibitem{Knutson} A. Knutson, K. Howell, 
Coupled Orbit and Attitude Dynamics for Spacecraft Comprised of Multiple Bodies in Earth-Moon Halo Orbits, 
\textit{IAF 63rd International Astronautical Congress}, October 1-5 2012.

\bibitem{Krener} A.J. Krener,
The high order maximal principle and its application to singular extremals, 
\textit{SIAM Journal on Control and Optimization}, 1977, vol. 15, no. 2, p. 256-293.

\bibitem{Kupka}  I. Kupka, 
Generalized Hamiltonians and optimal control: a geometric study of extremals. 
\textit{In Proceedings of the International Congress of Mathematicians}, 1986. p. 1180-1189.

\bibitem{Lara} M. Lara, J. Pelaez, C. Bombardelli, F. R. Lucas, M. Sanjurjo-Rivo, D. Curreli, E. C. Lorenzini, D.J. Scheeres, 
Dynamic Stabilization of L2 Periodic Orbits Using Attitude-Orbit Coupling Effects,
\textit{Journal of Aerospace Engineering}, 2012, vol. 4, no. 1, p. 73-82

\bibitem{Marchal} C. Marchal, 
Chattering arcs and chattering controls,
\textit{Journal of Optimization Theory and Applications}, 1973, vol. 11, no 5, p. 441-468.

\bibitem{Martinon} P. Martinon, J. Gergaud,
Using switching detection and variational equations for the shooting method,
\textit{Optimal Control Applications and Methods}, 2007, vol. 28, no. 2, p. 95-116.

\bibitem{Marec} J.P. Marec,
Optimal Space trajectories,
\textit{Elsevier}, 1979.

\bibitem{McDanell} J.P. McDanell, W.F. Powers,
Necessary conditions joining optimal singular and nonsingular subarcs,
\textit{SIAM Journal on Control}, 1971, vol. 9, no. 2, p. 161-173. 

\bibitem{More} J.J. Mor\'e, D.C. Sorensen, K.E. Hillstrom, et al. 
The MINPACK project. 
\textit{Sources and Development of Mathematical Software}, 1984, p. 88-111.

\bibitem{Pontryagin} L.S. Pontryagin,
Mathematical theory of optimal processes,
\textit{CRC Press}, 1987.

\bibitem{SCHATTLER} H. Sch\"{a}ttler, U. Ledzewicz,
Geometric Optimal Control: Theory, Methods and Examples,
\textit{Springer}, 2012. 

\bibitem{Seywald} H. Seywald, R.R. Kumar, 
Singular control in minimum time spacecraft reorientation,
\textit{Journal of Guidance, Control, and Dynamics}, 1993, vol. 16, no. 4, p. 686-694.

\bibitem{SHEN} H. Shen, P. Tsiotras,
Time-optimal control of axisymmetric rigid spacecraft using two controls,
\textit{Journal of Guidance, Control, and Dynamics s}, 1999, vol. 22, no. 5, p. 682-694.

\bibitem{Trelat2} E. Tr\'{e}lat,
Optimal control and applications to aerospace: some results and challenges,
\textit{Journal of Optimization Theory and Applications}, 2012, vol. 154, no. 3, p. 713-758.

\bibitem{TrelatZuazua_JDE2015}
E. Tr\'elat, E. Zuazua,
\textit{The turnpike property in finite-dimensional nonlinear optimal control},
J. Differential Equations {\bf 258} (2015), no. 1, 81--114.

\bibitem{ZELIKIN} M.I. Zelikin, V.F. Borisov, A.J. Krener,
Theory of Chattering Control: with applications to Astronautics, Robotics, Economics, and Engineering,
\textit{Springer}, 1994. 

\bibitem{ZTC}
J. Zhu, E. Tr\'elat, M. Cerf,
Planar tilting maneuver of a spacecraft: singular arcs in the minimum time problem and chattering,
Preprint arXiv:1504.06219 (2015), 43 pages.


\end{thebibliography}
\end{document}